\documentclass[10pt,reqno]{amsart}
\usepackage{amsmath,amssymb,latexsym,esint,cite,mathrsfs}
\usepackage{verbatim,wasysym,cite,relsize,color}
\usepackage[latin1]{inputenc}
\usepackage{microtype}
\usepackage{color,enumitem,graphicx}
\usepackage[colorlinks=true,urlcolor=blue, citecolor=red,linkcolor=blue,
linktocpage,pdfpagelabels, bookmarksnumbered,bookmarksopen]{hyperref}
\usepackage[hyperpageref]{backref}
\usepackage[english]{babel}
\usepackage{tikz}
\usepackage[font=small,skip=5pt]{caption}
\usepackage{geometry}
\numberwithin{equation}{section}
\newtheorem{theorem}{Theorem}[section]
\newtheorem{lemma}[theorem]{Lemma}
\newtheorem{remark}[theorem]{Remark}

\newtheorem{definition}[theorem]{Definition}
\newtheorem{proposition}[theorem]{Proposition}
\newtheorem{corollary}[theorem]{Corollary}

\newcommand{\R}{\mathbb R}

\newcommand{\vareps}{\varepsilon}
\newcommand{\eps}{\epsilon}
\DeclareMathOperator*{\loc}{loc}
\DeclareMathOperator*{\rad}{rad}
\DeclareMathOperator*{\opt}{opt}

\DeclareMathOperator*{\supp}{supp}
\DeclareMathOperator*{\gamc}{{\gamma_c}}
\DeclareMathOperator*{\sigc}{{\sigma_c}}
\DeclareMathOperator*{\ima}{Im}
\DeclareMathOperator*{\rea}{Re}
\newcommand{\scal}[1]{\left\langle #1 \right\rangle}

\title[Energy scattering 2D INLS]
{Energy scattering for a class of inhomogeneous nonlinear Schr\"odinger equation in two dimensions}
\author[V. D. Dinh]{Van Duong Dinh}
\address[V. D. Dinh]{Laboratoire Paul Painlev\'e UMR 8524, Universit\'e de Lille CNRS, 59655 Villeneuve d'Ascq Cedex, France
and 
Department of Mathematics, HCMC University of Pedagogy, 280 An Duong Vuong, Ho Chi Minh, Vietnam}
\email{contact@duongdinh.com}

\subjclass[2010]{35Q44; 35Q55}
\keywords{Inhomogeneous nonlinear Schr\"odinger equation, Scattering, Ground state, Radial Sobolev embedding}

\begin{document}
	
	\begin{abstract}
	We consider a class of $L^2$-supercritical inhomogeneous nonlinear Schr\"odinger equations in two dimensions
	\[
	i\partial_t u + \Delta u = \pm |x|^{-b} |u|^\alpha u, \quad (t,x) \in \R \times \R^2,
	\]
	where $0<b<1$ and $\alpha>2-b$. By adapting a new approach of Arora-Dodson-Murphy \cite{ADM}, we show the energy scattering for the equation with radially symmetric initial data. In the focusing case, our result extends the one of Farah-Guzm\'an \cite{FG-high} to the whole range of $b$ where the local well-posedness is available. In the defocusing case, our result extends the one in \cite{Dinh-scat} where the energy scattering for non-radial initial data was established in dimensions $N\geq 3$.
	\end{abstract}

	\maketitle

	\section{Introduction}
	\label{S1}
	\setcounter{equation}{0}
	We consider the Cauchy problem for the inhomogeneous nonlinear Schr\"odinger equations
	\begin{equation} \label{INLS}
		\left\{ 
		\begin{array}{rcl}
			i\partial_t u + \Delta u &=& \pm |x|^{-b} |u|^\alpha u, \quad (t,x) \in \R \times \R^N, \\
			u(0)&=& u_0,
		\end{array}
		\right.
	\end{equation}
	where $u: \mathbb{R} \times \mathbb{R}^N \rightarrow \mathbb{C}$, $u_0: \mathbb{R}^N \rightarrow \mathbb{C}$, $b>0$ and $\alpha>0$. The plus (resp. minus) sign in front of the nonlinearity corresponds to the defocusing (resp. focusing) case. The inhomogeneous nonlinear Schr\"odinger equation arises in nonlinear optics for the propagation of laser beams. The beam propagation can be modeled by the equation of the form
	\begin{align} \label{model-equ}
	i\partial_t u + \Delta u + K(x) |u|^\alpha u=0.
	\end{align}
	The equation \eqref{model-equ} has been attracted much attention recently. Berg\'e \cite{Berge} studied formally the stability condition for solition solutions of \eqref{model-equ}. Towers-Malomed \cite{TM} observed by means of variational approximation and direct simulations that a certain type of time-dependent nonlinear medium gives rise to completely stable beams. Merle \cite{Merle} and Rapha\"el-Szeftel \cite{RS} studied the existence and non-existence of minimal mass blow-up solutions  for \eqref{model-equ} with $k_1 < K(x) <k_2$ and $k_1, k_2>0$. Fibich-Wang \cite{FW} investigated the stability of solitary waves for \eqref{model-equ} with $K(x)= K(\eps |x|)$, where $\eps>0$ is a small parameter and $K \in C^4(\R^N) \cap L^\infty(\R^N)$. The case $K(x) = |x|^b$ with $b>0$ was studied in \cite{Chen, CG, LWW, Zhu}.
	
	Before reviewing some known results for \eqref{INLS}, let us recall some properties of \eqref{INLS}. The equation \eqref{INLS} is invariant under the scaling
	\[
	u_\lambda(t,x):= \lambda^{\frac{2-b}{\alpha}} u(\lambda^2 t, \lambda x), \quad \lambda>0.
	\]
	An easy computation shows 
	\[
	\|u_\lambda(0)\|_{\dot{H}^\gamma} = \lambda^{\gamma-\frac{N}{2}+\frac{2-b}{\alpha}} \|u_0\|_{\dot{H}^\gamma}.
	\]
	We thus denote the critical exponents
	\[
	\gamc := \frac{N}{2} - \frac{2-b}{\alpha}
	\]
	and 
	\begin{align} \label{def-sigc}
	\sigc:= \frac{1-\gamc}{\gamc} = \frac{4-2b-(N-2)\alpha}{N\alpha-4+2b}.
	\end{align}
	The equation \eqref{INLS} has formally the conservation of mass and energy
	\begin{align*}
	M(u(t)) &= \int |u(t,x)|^2 dx = M(u_0), \tag{\text{Mass}} \\
	E(u(t)) &= \frac{1}{2} \int |\nabla u(t,x)|^2 dx \pm \frac{1}{\alpha+2} \int |x|^{-b} |u(t,x)|^{\alpha+2} dx = E(u_0). \tag{\text{Energy}}
	\end{align*}
	The well-posedness for \eqref{INLS} with initial data in $H^1$ was first studied by Genoud-Stuart \cite{GS} by using an abstract theory of Cazenave which does not use Strichartz estimates. More precisely, they proved that the focusing problem \eqref{INLS} with $0<b<\min\{2,N\}$ is well posed in $H^1$:
	\begin{itemize}
		\item locally if $0<\alpha<2^*$,
		\item globally for any initial data if $0<\alpha<2_*$,
		\item globally for small initial data if $2_* \leq \alpha<2^*$,
	\end{itemize}
	where
	\begin{align} \label{def-2-star}
	2^*:= \left\{
	\begin{array}{cl}
	\frac{4-2b}{N-2} & \text{if } N\geq 3, \\
	\infty &\text{if } N=1,2,
	\end{array}
	\right.
	\quad
	2_*:= \frac{4-2b}{N}.	
	\end{align}
	
	Guzm\'an \cite{Guzman} and Dinh \cite{Dinh-weigh} later used Strichartz estimates and the contraction mapping argument to show the local well-posedness for \eqref{INLS}. They proved that if
	\begin{align*}
	\renewcommand*{\arraystretch}{1.3}
	\left\{
	\begin{array}{l}
	N\geq 4, \quad 0<b<2, \quad 0<\alpha<2^*, \\
	N=3, \quad 0<b<1, \quad 0<\alpha<2^*, \\
	N=3, \quad 1 \leq b <\frac{3}{2}, \quad 0<\alpha<\frac{6-4b}{2b-1}, \\
	N=2, \quad 0<b<1, \quad 0<\alpha<2^*,
	\end{array}
	\right.
	\end{align*}
	then \eqref{INLS} is locally well-posed in $H^1$. Moroever, the local solution satisfies $u \in L^q_{\loc}((-T_*, T^*), W^{1,r})$ for any Schr\"odinger admissible pair $(q,r)$, where $(-T_*, T^*)$ is the maximal time of existence. Note that the results of Guzm\'an and Dinh are weaker than the ones of Genoud and Stuart. It does not treat the case $N=1$ and there are restrictions on the validity of $b$ when $N=2$ and $N=3$. However, it shows that the solution belongs locally in Strichartz spaces $L^q((-T_*, T^*), W^{1,r})$.  This property plays a crucial role in the scattering theory. 
	
	In the case $\alpha=2_*$, Genoud \cite{Genoud} showed that the focusing problem \eqref{INLS} with $0<b<\min \{2,N\}$ is globally well-posed in $H^1$ by assuming $u_0 \in H^1$ and $\|u_0\|_{L^2} <\|Q\|_{L^2}$, where $Q$ is the unique positive radially symmetric and decreasing solution to the elliptic equation
	\begin{align} \label{ell-equ}
	\Delta Q-Q+|x|^{-b} |Q|^\alpha Q=0.
	\end{align}
	Combet-Genoud \cite{CG} later established the classification of minimal mass blow-up solutions to the focusing problem \eqref{INLS}. Note that the uniqueness of positive radial solution to \eqref{ell-equ} was established by Yanagida \cite{Yanagida} and Genoud \cite{Genoud-2d}. Their results hold under the assumptions $0<b<\min \{2,N\}$ and $0<\alpha<2^*$. 
	
	In the case $2_*<\alpha<2^*$, Farah \cite{Farah} proved that the focusing problem \eqref{INLS} with $0<b<\min \{2,N\}$ is globally well-posed in $H^1$ provided that $u_0 \in H^1$ and satisfies
	\begin{align} \label{con-energy}
	E(u_0) [M(u_0)]^{\sigc} < E(Q) [M(Q)]^{\sigc}
	\end{align}
	and
	\begin{align} \label{con-grad}
	\|\nabla u_0\|_{L^2} \|u_0\|_{L^2}^{\sigc} < \|\nabla Q\|_{L^2} \|Q\|^{\sigc}_{L^2},
	\end{align}
	where $\sigc$ is as in \eqref{def-sigc}. The existence of finite time blow-up solutions for the focusing problem \eqref{INLS} was studied by Farah \cite{Farah} and Dinh \cite{Dinh-blow}. 
	
	The energy scattering for the focusing problem \eqref{INLS} was first established by Farah-Guzm\'an \cite{FG-3D} with $0<b<1$, $\alpha=2$ and $N=3$. The proof is based on the concentration-compactness argument developed by Kenig-Merle \cite{KM}. This result was later extended to higher dimensions in \cite{FG-high} using again the concentration-compactness argument. Recently, Campos \cite{Campos} used a new method of Dodson-Murphy \cite{DM} to give an alternative simple proof for the results of Farah-Guzm\'an. He also extends the validity of $b$ in dimensions $N\geq 3$. More precisely, their results read as follows.
	
	\begin{theorem}[\cite{FG-high,Campos}] 
		Let $N\geq 3$, $0<b<2$ and $2_*<\alpha<2^*$. Let $u_0 \in H^1$ be radially symmetric and satisfy \eqref{con-energy} and \eqref{con-grad}. Then the corresponding solution to the focusing problem \eqref{INLS} exists globally in time and scatters in both time directions, i.e. there exist $u_0^\pm \in H^1$ such that
		\[
		\lim_{t\rightarrow \pm \infty} \|u(t) - e^{it\Delta} u_0^\pm \|_{H^1} =0.
		\]
	\end{theorem}

	In the case $N=2$, we also have the following energy scattering for the focusing problem \eqref{INLS} due to Farah-Guzm\'an \cite{FG-high}.
	
	\begin{theorem}[\cite{FG-high}] 
		Let $N=2$, $0<b<\frac{2}{3}$ and $\alpha>2-b$. Let $u_0 \in H^1$ be radially symmetric and satisfy \eqref{con-energy} and \eqref{con-grad}. Then the corresponding solution to the focusing problem \eqref{INLS} exists globally in time and scatters in both time directions.
	\end{theorem}

	The main purpose of this paper is to give an alternative simple proof for the result of Farah-Guzm\'an in two dimensions. More precisely, our main result is the following.
	
	\begin{theorem} \label{theo-scat-2D-focus}
		Let $N=2$, $0<b<1$ and $\alpha>2-b$. Let $u_0 \in H^1$ be radially symmetric and satisfy \eqref{con-energy} and \eqref{con-grad}. Then the corresponding solution to the focusing problem \eqref{INLS} exists globally in time and scatters in both time directions.
	\end{theorem}

	\begin{remark}
		Our result extends the one of \cite{FG-high} to the whole range of $b$ where the local well-posedness is available. 
	\end{remark}

	The proof of Theorem $\ref{theo-scat-2D-focus}$ is based on a recent argument of Arora-Dodson-Murphy \cite{ADM}. Due to the radially symmetric property of the solution, we first derive Morawetz estimates related to the solution. As a consequence, we get the space time estimates
	\[
	\int_0^T \int |x|^{-b}|u(t,x)|^{\alpha+2} dx dt \leq C(u_0,Q) T^{\beta_1}, \quad \beta_1:= \max\left\{ \frac{1}{3}, \frac{2}{\alpha+2+2b}\right\}
	\]
	for any $T>0$ sufficiently large and
	\[
	\int_I \int |u(t,x)|^{\alpha+2+b} dx dt \leq C(u_0,Q) |I|^{\beta_2}, \quad \beta_2:= \max \left\{\frac{2+b}{6}, \frac{2+b}{\alpha+2+2b} \right\}
	\]
	for any time interval $I \subset \R$. Note that $0<\beta_1+\beta_2 <1$ since $0<b<1$ and $\alpha>2-b$. Using the above space time estimates, we show the global bound $\|u\|_{L^{\alpha+2+b}(\R \times \R^2)} \leq C(u_0,Q)<\infty$ which implies the energy scattering. We refer the reader to Section $\ref{S3}$ for more details.

	\begin{remark}
		After finishing the manuscript, we learn that Xu-Zhao \cite{XZ} has simultaneously proved the same result as Theorem $\ref{theo-scat-2D-focus}$. 
	\end{remark}

	In the defocusing case, the energy scattering for \eqref{INLS} was first established in \cite{Dinh-weigh} by considering the initial data in the weighted $L^2$ space $\Sigma:= H^1 \cap L^2(|x|^2 dx)$. The energy scattering for the defocusing problem \eqref{INLS} with initial data in $H^1$ in dimensions $N\geq 3$ was proved by the author in \cite{Dinh-scat}. The proof is based on the decay property of global solutions. We refer the reader to Appendix $\ref{S4}$ for an alternative proof which makes use of the interaction Morawetz inequality. Our contribution in this direction is the following energy scattering for the defocusing \eqref{INLS} in 2D with radially symmetric initial data. 
	
	\begin{theorem} \label{theo-scat-2D-defocus}
		Let $N=2$, $0<b<1$ and $\alpha>2-b$. Let $u_0 \in H^1$ be radially symmetric. Then the corresponding solution to the defocusing problem \eqref{INLS} exists globally in time and scatters in both directions.
	\end{theorem}

	This paper is organized as follows. In Section $\ref{S2}$, we  give some preliminaries including Strichartz estimates, some variational analysis and Morawetz estimates related to the equation. In Section $\ref{S3}$, we give the proofs of the energy scattering given in Theorem $\ref{theo-scat-2D-focus}$ and Theorem $\ref{theo-scat-2D-defocus}$. Finally, an alternative proof of the energy scattering for the defocusing  problem \eqref{INLS} in dimensions $N\geq 3$ is given in the Appendix.
	
	\section{Preliminaries}
	\label{S2}
	\setcounter{equation}{0}
	\subsection{Strichartz estimates}
	Let $I\subset \R$ and $q,r \in [1,\infty]$. We define the mixed norm 
	\[
	\|u\|_{L^q(I,L^r)} := \left( \int_I \left( \int_{\R^N} |u(t,x)|^r dx \right)^{\frac{q}{r}} dt \right)^{\frac{1}{q}}
	\]
	with a usual modification when either $q$ or $r$ are infinity. When $q=r$, we use the notation $L^q(I \times \R^N)$ instead of $L^q(I,L^q)$.
	\begin{definition}
		A pair $(q,r)$ is said to be Schr\"odinger admissible, for short $(q,r) \in S$, if
		\[
		(q,r) \in [2,\infty]^2, \quad (q,r,N) \ne (2,\infty,2), \quad \frac{2}{q}+\frac{N}{r} =\frac{N}{2}.
		\]
	\end{definition}
	For any interval $I\subset \R$, we denote the Strichartz norm
	\begin{align} \label{str-cha-norm}
	\|u\|_{S(L^2,I)} :=\sup_{(q,r) \in S} \|u\|_{L^q(I, L^r)}, \quad \|v\|_{S'(L^2,I)} := \sup_{(q,r)\in S} \|v\|_{L^{q'}(I,L^{r'})},
	\end{align}
	where $(q,q')$ and $(r,r')$ are H\"older conjugate pairs. 
	
	We next recall the well-known Strichartz estimates for the linear Schr\"odinger equation (see e.g. \cite{Cazenave, Tao}).
	\begin{proposition}[\cite{Cazenave, Tao}] 
		Let $u$ be a solution to the linear Schr\"odinger equation, namely
		\[
		u(t) = e^{it\Delta} u_0 + \int_0^t e^{i(t-s)\Delta} F(s) ds
		\]
		for some data $u_0$ and $F$. Then it holds that
		\[
		\|u\|_{S(L^2,\R)} \lesssim \|u_0\|_{L^2} + \|F\|_{S'(L^2,\R)}.
		\]
	\end{proposition}
	
	\subsection{Variational analysis}
	Let us recall some properties related to the ground state $Q$  which is the unique positive radial decreasing solution to
	\[
	\Delta Q - Q + |x|^{-b} |Q|^\alpha Q =0.
	\]
	The ground state $Q$ optimizes the following Gagliardo-Nirenberg inequality: $N\geq 1$, $0<b<\min\{2,N\}$ and $0<\alpha <2^*$ (see \eqref{def-2-star}),
	\[
	\int |x|^{-b} |f(x)|^{\alpha+2} dx \leq C_{\opt} \|f\|^{\frac{4-2b - (N-2)\alpha}{2}}_{L^2} \|\nabla f\|^{\frac{N\alpha+2b}{2}}_{L^2}, \quad f \in H^1,
	\]
	that is,
	\[
	C_{\opt} = \int |x|^{-b} |Q(x)|^{\alpha+2} dx \div \left[\|Q\|^{\frac{4-2b - (N-2)\alpha}{2}}_{L^2} \|\nabla Q\|^{\frac{N\alpha+2b}{2}}_{L^2} \right].
	\]
	It was shown in \cite{Farah} that $Q$ satisfies the following Pohozaev's identities
	\[
	\|Q\|^2_{L^2} = \frac{4-2b-(N-2)\alpha}{N\alpha+2b} \|\nabla Q\|^2_{L^2} = \frac{4-2b-(N-2)\alpha}{2(\alpha+2)} \int |x|^{-b} |Q(x)|^{\alpha+2} dx.
	\]
	In particular,
	\[
	C_{\opt} = \frac{2(\alpha+2)}{N\alpha+2b} \left( \|\nabla Q\|_{L^2} \|Q\|_{L^2}^{\sigc} \right)^{-\frac{N\alpha-4+2b}{2}},
	\]
	where $\sigc$ is defined in \eqref{def-sigc}.
	
	\begin{lemma}[\cite{FG-high}] \label{lem-coer-1}
		Let $N\geq 1$, $0<b<\min \{2,N\}$ and $2_*<\alpha <2^*$. Let $u_0 \in H^1$ satisfy \eqref{con-energy} and \eqref{con-grad}. Then the corresponding solution to the focusing problem \eqref{INLS} satisfies
		\begin{align} \label{glo-bou-1}
		\|\nabla u(t)\|_{L^2} \|u(t)\|^{\sigc}_{L^2} < \|\nabla Q\|_{L^2} \|Q\|^{\sigc}_{L^2}
		\end{align}
		for all $t$ in the existence time. In particular, the corresponding solution to the focusing problem \eqref{INLS} exists globally in time. Moreover, there exists $\rho=\rho(u_0,Q)>0$ such that
		\begin{align} \label{glo-bou-2}
		\|\nabla u(t)\|_{L^2} \|u(t)\|^{\sigc}_{L^2}< (1-2\rho)\|\nabla Q\|_{L^2} \|Q\|^{\sigc}_{L^2}
		\end{align}
		for all $t \in \R$.
	\end{lemma}
	We refer the reader to \cite[Lemma 4.2]{FG-high} for the proof of this result.
	
	\begin{lemma}[\cite{Campos}] \label{lem-coer-2}
		Let $N\geq 1$, $0<b<\min \{2,N\}$ and $2_*<\alpha <2^*$. Let $u_0 \in H^1$ satisfy \eqref{con-energy} and \eqref{con-grad}. Let $\rho$ be as in Lemma $\ref{lem-coer-1}$. Then there exists $R_0=R_0(\rho, u_0)>0$ such that for any $R \geq R_0$,
		\begin{align} \label{glo-bou-3}
		\|\nabla(\chi_R u(t))\|_{L^2} \|\chi_R u(t)\|^{\sigc}_{L^2} <(1-\rho) \|\nabla Q\|_{L^2} \|Q\|^{\sigc}_{L^2}
		\end{align}
		for all $t\in \R$, where $\chi_R(x) = \chi(x/R)$ with $\chi \in C^\infty_0(\R^N)$ satisfying $0\leq \chi \leq 1$, $\chi=1$ on $B(0,1/2)$ and $\chi=0$ on $B^c(0,1)$. Moreover, there exists $\delta=\delta(\rho)>0$ such that
		\begin{align} \label{lower-bou-virial}
		\|\nabla(\chi_R u(t))\|^2_{L^2} - \frac{N\alpha+2b}{2(\alpha+2)} \int |x|^{-b} |\chi_R u(t,x)|^{\alpha+2} dx \geq \delta \int |x|^{-b} |\chi_R u(t,x)|^{\alpha+2} dx
		\end{align}
		for all $t\in \R$.
	\end{lemma}
	We refer the reader to \cite[Lemma 4.4]{Campos} for the proof of this result.
	
	\subsection{Morawetz estimate}
	Let us start with the following virial identity.
	\begin{lemma}[Virial identity \cite{Farah, Dinh-blow}] \label{lem-virial-iden}
		Let $N\geq 1$, $0<b<\min \{2,N\}$ and $0<\alpha<2^*$. Let $\varphi: \R^N \rightarrow \R$ be a sufficiently smooth and decaying function. Let $u$ be a solution to \eqref{INLS}. Define
		\[
		M_\varphi(t):= 2 \int \nabla \varphi \cdot \ima \left( \overline{u}(t) \nabla u(t)\right) dx.
		\]
		Then it holds that
		\begin{align*}
		\frac{d}{dt} M_\varphi(t) = -\int \Delta^2 \varphi |u(t)|^2 dx &+ 4 \sum_{j,k=1}^N \int \partial^2_{jk} \varphi \rea \left( \partial_j \overline{u}(t) \partial_k u(t)\right) dx \\
		&\pm \frac{2\alpha}{\alpha+2} \int |x|^{-b} \Delta \varphi |u(t)|^{\alpha+2} dx \pm \frac{4b}{\alpha+2} \int |x|^{-b-2} x \cdot \nabla \varphi |u(t)|^{\alpha+2} dx.
		\end{align*}
	\end{lemma}
	We now define a non-negative function $\varphi: [0,\infty) \rightarrow [0,\infty)$ satisfying
	\[
	\varphi (r) = \left\{
	\begin{array}{cl}
	r^2 &\text{if } 0 \leq r \leq 1, \\
	\text{smooth} &\text{if } 1<r<2, \\
	2 &\text{if } r\geq 2,
	\end{array}
	\right.
	\quad \varphi'(r) \geq 0, \quad 0 \leq \varphi''(r) \leq 2 \text{ for any } r\geq 0.
	\]
	Given $R>0$, we define a radial function
	\begin{align} \label{def-varphi-R}
	\varphi_R(x) = \varphi_R(r):= R^2 \varphi(r/R), \quad r=|x|.
	\end{align}
	It is easy to check that
	\[
	2-\varphi''_R(r) \geq 0, \quad 2-\frac{\varphi'(r)}{r} \geq 0, \quad 2N-\Delta \varphi_R(x) \geq 0, \quad \forall r \geq 0, x \in \R^N.
	\]
	We also have that
	\[
	\|\nabla^k \varphi_R\|_{L^\infty} \lesssim R^{2-k}, \quad k=0,\cdots, 4
	\]
	and
	\[
	\supp(\nabla^k \varphi_R) \subset \left\{
	\begin{array}{cl}
	\{|x| \leq 2R\} &\text{if } k=1,2, \\
	\{R \leq |x| \leq 2R\} &\text{if } k=3,4.
	\end{array}
	\right.
	\]
	
	\begin{proposition} \label{prop-mora-est}
		Let $N\geq 2$, $0<b<2$ and $2_*<\alpha<2^*$. Let $u_0 \in H^1$ be radially symmetric and satisfy \eqref{con-energy} and \eqref{con-grad}. Then for any $T>0$ sufficiently large, the corresponding global solution to the focusing problem \eqref{INLS} satisfies
		\begin{align} \label{mora-est}
		\int_0^T \int |x|^{-b} |u(t,x)|^{\alpha+2} dx dt \leq C(u_0,Q) T^{\beta_1}, \quad \beta_1:= \max \left\{\frac{1}{3}, \frac{2}{(N-1)\alpha+2+2b}\right\}
		\end{align}
		for some constant $C(u_0,Q)$ depending only on $u_0$ and $Q$. Moroever, for any interval $I \subset \R$,
		\begin{align} \label{mora-est-I}
		\int_I \|u(t)\|^{\alpha+2+\frac{b}{N-1}}_{L^{\alpha+2+\frac{b}{N-1}}} dt \leq C(u_0,Q) |I|^{\beta_2}, \quad \beta_2:= \max \left\{\frac{2+b}{6}, \frac{2+b}{(N-1)\alpha +2+2b} \right\}.
		\end{align}
	\end{proposition}

	\begin{proof}
		Let $\rho=\rho(u_0,Q)$ be as in \eqref{glo-bou-2}, and $R_0=R_0(\rho,u_0)$ be as in Lemma $\ref{lem-coer-2}$. Let $\varphi_R$ be as in \eqref{def-varphi-R}. By the Cauchy-Schwarz inequality, the conservation of mass and \eqref{glo-bou-1}, we see that
		\begin{align} \label{mora-est-proof}
		|M_{\varphi_R}(t)| \leq \|\nabla \varphi_R\|_{L^\infty} \|u(t)\|_{L^2} \|\nabla u(t)\|_{L^2} \lesssim R
		\end{align}
		for all $t\in \R$. By Lemma $\ref{lem-virial-iden}$ and the fact $\varphi_R(x)=|x|^2$ for $|x| \leq R$, 
		\begin{align*}
		\frac{d}{dt} M_{\varphi_R}(t) &= - \int \Delta^2 \varphi_R |u(t)|^2 dx + 4 \sum_{j,k=1}^N \int \partial^2_{jk} \varphi_R \rea \left( \partial_j \overline{u}(t) \partial_k u(t) \right) dx \\
		& \mathrel{\phantom{= - \int \Delta^2 \varphi_R |u(t)|^2 dx}} - \frac{2\alpha}{\alpha+2} \int |x|^{-b} \Delta \varphi_R |u(t)|^{\alpha+2} dx - \frac{4b}{\alpha+2} \int |x|^{-b-2} x \cdot \nabla \varphi_R |u(t)|^{\alpha+2} dx \\
		&= 8 \left( \int_{|x| \leq R} |\nabla u(t)|^2 dx - \frac{N\alpha+2b}{2(\alpha+2)} \int_{|x| \leq R} |x|^{-b} |u(t)|^{\alpha+2} dx \right) \\
		&\mathrel{\phantom{=}} - \int \Delta^2 \varphi_R |u(t)|^2 dx + 4 \sum_{j,k=1}^N \int_{R \leq |x| \leq 2R} \partial^2_{jk} \varphi_R \rea \left( \partial_j \overline{u}(t) \partial_k u(t) \right) dx \\
		&\mathrel{\phantom{=}} - \frac{2\alpha}{\alpha+2} \int_{R \leq |x| \leq 2R} |x|^{-b} \Delta \varphi_R |u(t)|^{\alpha+2} dx - \frac{4b}{\alpha+2} \int_{R \leq |x| \leq 2R} |x|^{-b-2} x \cdot \nabla \varphi_R |u(t)|^{\alpha+2} dx.
		\end{align*}
		Since $\|\Delta^2 \varphi_R\|_{L^\infty} \lesssim R^{-2}$, the conservation of mass implies
		\[
		\int \Delta^2 \varphi_R |u(t)|^2 dx \lesssim R^{-2}.
		\]
		Since $u$ is radial, we use the fact
		\[
		\partial^2_{jk} = \left(\frac{\delta_{jk}}{r} - \frac{x_j x_k}{r^3} \right) \partial_r + \frac{x_j x_k}{r^2} \partial^2_r
		\]
		to get
		\[
		\sum_{j,k=1}^N \partial^2_{jk} \varphi_R \partial_j \overline{u} \partial_k u = \varphi''_R |\partial_r u|^2 \geq 0
		\]
		which implies
		\[
		\int_{R \leq |x| \leq 2R} \partial^2_{jk} \varphi_R \rea \left( \partial_j \overline{u}(t) \partial_k u(t) \right) dx \geq 0.
		\]
		Since $\|\Delta \varphi_R\|_{L^\infty} \lesssim 1$ and $\|x \cdot \nabla \varphi_R\|_{L^\infty} \lesssim |x|^2$, the radial Sobolev embedding (see e.g. \cite{Strauss}): $N\geq 2$,
		\begin{align} \label{rad-sob-emb}
		\||x|^{\frac{N-1}{2}} f\|_{L^\infty} \lesssim \|f\|_{H^1}, \quad \forall f \in H^1_{\rad}
		\end{align}
		implies that
		\begin{align*}
		\left|\int_{R \leq |x| \leq 2R} \left(|x|^{-b} \Delta \varphi_R + |x|^{-b-2} x\cdot \nabla \varphi_R\right) |u(t)|^{\alpha+2} dx \right| &\lesssim \left(\sup_{|x|\geq R} |x|^{-b} |u(t,x)|^\alpha \right) \|u(t)\|^2_{L^2} \\
		&\lesssim R^{- \frac{(N-1)\alpha}{2} -b} \|u(t)\|^\alpha_{H^1} \|u(t)\|^2_{L^2} \lesssim R^{-\frac{(N-1)\alpha}{2}-b}.
		\end{align*}
		It follows that
		\[
		\frac{d}{dt} M_{\varphi_R}(t) \geq 8 \left(\int_{|x| \leq R} |\nabla u(t)|^2 dx - \frac{N\alpha+2b}{2(\alpha+2)} \int_{|x| \leq R} |x|^{-b} |u(t)|^{\alpha+2} dx \right) + O \left( R^{-2} + R^{-\frac{(N-1)\alpha}{2} -b}\right).
		\]
		On the other hand, let $\chi_R$ be as in Lemma $\ref{lem-coer-2}$. We see that
		\begin{align*}
		\int |\nabla(\chi_R u(t))|^2 dx &= \int \chi^2_R |\nabla u(t)|^2 dx - \int \chi_R \Delta(\chi_R) |u(t)|^2 dx \\
		&= \int_{|x| \leq R} |\nabla u(t)|^2 dx - \int_{R/2 \leq |x| \leq R} (1-\chi^2_R) |\nabla u(t)|^2 dx - \int \chi_R \Delta(\chi_R) |u(t)|^2 dx
		\end{align*}
		and
		\[
		\int |x|^{-b}|\chi_R u(t)|^{\alpha+2} dx = \int_{|x| \leq R} |x|^{-b} |u(t)|^{\alpha+2} dx - \int_{R/2 \leq |x| \leq R} (1-\chi_R^{\alpha+2}) |x|^{-b} |u(t)|^{\alpha+2} dx.
		\]
		It follows that
		\begin{align*}
		\int_{|x| \leq R} |\nabla u(t)|^2 dx &- \frac{N\alpha+2b}{2(\alpha+2)} \int_{|x| \leq R} |x|^{-b} |u(t)|^{\alpha+2} dx \\
		&= \int |\nabla(\chi_R u(t))|^2 dx - \frac{N\alpha+2b}{2(\alpha+2)} \int |x|^{-b} |\chi_R u(t)|^{\alpha+2} dx \\
		&\mathrel{\phantom{=}} + \int_{R/2 \leq |x| \leq R} (1-\chi^2_R) |\nabla u(t)|^2 dx + \int \chi_R \Delta (\chi_R) |u(t)|^2 dx \\
		&\mathrel{\phantom{=}} - \frac{N\alpha+2b}{2(\alpha+2)} \int_{R/2 \leq |x| \leq R} (1-\chi_R^{\alpha+2}) |x|^{-b} |u(t)|^{\alpha+2} dx.
		\end{align*}
		Thanks to the fact $0\leq \chi_R \leq 1$, $\|\Delta(\chi_R)\|_{L^\infty} \lesssim R^{-2}$ and the radial Sobolev embedding, we infer that
		\begin{align*}
		\int_{|x| \leq R} |\nabla u(t)|^2 dx &- \frac{N\alpha+2b}{2(\alpha+2)} \int_{|x| \leq R} |x|^{-b} |u(t)|^{\alpha+2} dx \\
		&\geq \int |\nabla(\chi_R u(t))|^2 dx - \frac{N\alpha+2b}{2(\alpha+2)} \int |x|^{-b} |\chi_R u(t)|^{\alpha+2} dx + O\left(R^{-2} + R^{-\frac{(N-1)\alpha}{2} -b} \right).
		\end{align*}
		We thus obtain
		\[
		\frac{d}{dt} M_{\varphi_R}(t) \geq 8 \left( \int |\nabla(\chi_R u(t))|^2 dx - \frac{N\alpha+2b}{2(\alpha+2)} \int |x|^{-b} |\chi_R u(t)|^{\alpha+2} dx \right) + O\left(R^{-2} + R^{-\frac{(N-1)\alpha}{2} -b} \right).
		\]
		By Lemma $\ref{lem-coer-2}$, there exist $\delta=\delta(\rho)>0$ and $R_0=R_0(\rho,u_0)>0$ such that for any $R\geq R_0$,
		\[
		8 \delta \int |x|^{-b} |\chi_R u(t)|^{\alpha+2} dx \leq \frac{d}{dt} M_{\varphi_R}(t) + O\left(R^{-2} + R^{-\frac{(N-1)\alpha}{2} -b} \right)
		\]
		which, by \eqref{mora-est-proof}, implies 
		\[
		8 \delta \int_0^T \int |x|^{-b} |\chi_R u(t)|^{\alpha+2} dx dt \leq R + O\left(R^{-2} + R^{-\frac{(N-1)\alpha}{2} -b} \right) T.
		\]
		By the definition of $\chi_R$, 
		\begin{align} \label{est-T}
		\int_0^T \int_{|x| \leq R/2} |x|^{-b} |u(t,x)|^{\alpha+2} dx dt \lesssim R + \frac{T}{R^2} + \frac{T}{R^{\frac{(N-1)\alpha}{2} +b}}.
		\end{align}
		On the other hand,
		\[
		\int_{|x|\geq R/2} |x|^{-b} |u(t,x)|^{\alpha+2} dx \leq \left( \sup_{|x|\geq R/2} |x|^{-b} |u(t,x)|^\alpha \right) \|u(t)\|^2_{L^2} \lesssim \frac{1}{R^{\frac{(N-1)\alpha}{2}+b}}.
		\]
		We thus get
		\[
		\int_0^T \int |x|^{-b} |u(t,x)|^{\alpha+2} dx dt \lesssim R + \frac{T}{R^2} + \frac{T}{R^{\frac{(N-1)\alpha}{2} +b}}
		\]
		which proves \eqref{mora-est} by taking 
		\[
		R=T^{\frac{1}{1+ \min \left\{2, \frac{(N-1)\alpha}{2} +b \right\}}} = T^{\beta_1}.
		\]
		As in \eqref{est-T}, we also have for any interval $I \subset \R$,
		\[
		\int_I \int_{|x|\leq R/2} |x|^{-b} |u(t,x)|^{\alpha+2} dx dt \lesssim R + \frac{|I|}{R^2} + \frac{|I|}{R^{\frac{(N-1)\alpha}{2}+b}}
		\]
		hence
		\[
		\int_I \int_{|x| \leq R/2} |x|^{-\frac{b}{2}} |u(t,x)|^{\alpha+2} dx dt \lesssim R^{\frac{b}{2}} \int_I \int_{|x|\leq R/2} |x|^{-b} |u(t,x)|^{\alpha+2} dx dt \lesssim R^{1+\frac{b}{2}} + \frac{|I|}{R^{2-\frac{b}{2}}} + \frac{|I|}{R^{\frac{(N-1)\alpha+b}{2}}}.
		\]
		We also have
		\[
		\int_{|x|\geq R/2} |x|^{-\frac{b}{2}} |u(t,x)|^{\alpha+2} dx dt 
		\lesssim \frac{1}{R^{\frac{(N-1)\alpha+b}{2}}}.
		\]
		It follows that
		\[
		\int_I \int |x|^{-\frac{b}{2}} |u(t,x)|^{\alpha+2} dx dt \lesssim R^{1+\frac{b}{2}} + \frac{|I|}{R^\sigma},
		\]
		where 
		\[
		\sigma:= \min \left\{2-\frac{b}{2}, \frac{(N-1)\alpha+b}{2} \right\}.
		\]
		Taking $R=|I|^{\frac{1}{1+\frac{b}{2}+\sigma}}$, we get for $|I|$ sufficiently large,
		\[
		\int_I \int |x|^{-\frac{b}{2}} |u(t,x)|^{\alpha+2} dx dt \lesssim |I|^{\frac{1+\frac{b}{2}}{1+\frac{b}{2}+\sigma}} = |I|^{\beta_2}
		\]
		with $\beta_2$ as in \eqref{mora-est-I}. By the radial Sobolev embedding \eqref{rad-sob-emb}, 
		\begin{align*}
		\int_I \|u(t)\|^{\alpha+2+\frac{b}{N-1}}_{L^{\alpha+2+\frac{b}{N-1}}} dt &= \int_I \int \left( |x|^{\frac{N-1}{2}} |u(t,x)|\right)^{\frac{b}{N-1}} |x|^{-\frac{b}{2}} |u(t,x)|^{\alpha+2} dx dt \\
		&\lesssim \int_I \|u(t)\|^{\frac{b}{N-1}}_{H^1} |x|^{-\frac{b}{2}}|u(t,x)|^{\alpha+2} dx dt \lesssim |I|^{\beta_2}
		\end{align*}
		which proves \eqref{mora-est-I} for $|I|$ sufficiently large. In the case $|I|$ is sufficiently small, it follows from the Sobolev embedding \footnote{It is easy to check that $\alpha+2 + \frac{b}{N-1} \leq \frac{2N}{N-2}$ if $N\geq 3$ since $\alpha < \frac{4-2b}{N-2}$.} $\|u(t)\|_{L^{\alpha+2+\frac{b}{N-1}}} \lesssim \|u(t)\|_{H^1}$ that
		\[
		\int_I \|u(t)\|^{\alpha+2+\frac{b}{N-1}}_{L^{\alpha+2+\frac{b}{N-1}}} dt \lesssim |I| \lesssim |I|^{\beta_2}
		\]
		since $\beta_2<1$. The proof is complete.
	\end{proof}

	\begin{corollary} \label{coro-mora-est}
		Let $N\geq 2$, $0<b<2$ and $2_*<\alpha<2^*$. Let $u_0 \in H^1$ be radially symmetric. Then for any $T>0$ sufficiently large, the corresponding global solution to the defocusing problem \eqref{INLS} satisfies
		\[
		\int_0^T \int |x|^{-b} |u(t,x)|^{\alpha+2} dx dt \leq C(u_0) T^{\beta_1}
		\]
		for some constant $C(u_0)$ depending only on the mass and energy of the initial data $u_0$, where $\beta_1$ is as in \eqref{mora-est}. Moroever, for any interval $I \subset \R$,
		\begin{align} \label{mora-est-I-defocus}
		\int_I \|u(t)\|^{\alpha+2+\frac{b}{N-1}}_{L^{\alpha+2+\frac{b}{N-1}}} dt \leq C(u_0) |I|^{\beta_2},
		\end{align}
		where $\beta_2$ is given in \eqref{mora-est-I}.
	\end{corollary}
		
	\begin{proof}
		The proof is similar to the one of Proposition $\ref{prop-mora-est}$. We only point out the differences. We first have
		\begin{align*}
		\frac{d}{dt} M_{\varphi_R}(t) &= - \int \Delta^2 \varphi_R |u(t)|^2 dx + 4 \sum_{j,k=1}^N \int \partial^2_{jk} \varphi_R \rea \left( \partial_j \overline{u}(t) \partial_k u(t) \right) dx \\
		& \mathrel{\phantom{= - \int \Delta^2 \varphi_R |u(t)|^2 dx}} + \frac{2\alpha}{\alpha+2} \int |x|^{-b} \Delta \varphi_R |u|^{\alpha+2} dx + \frac{4b}{\alpha+2} \int |x|^{-b-2} x \cdot \nabla \varphi_R |u|^{\alpha+2} dx \\
		&= 8 \left( \int_{|x| \leq R} |\nabla u(t)|^2 dx + \frac{N\alpha+2b}{2(\alpha+2)} \int_{|x| \leq R} |x|^{-b} |u(t)|^{\alpha+2} dx \right) \\
		&\mathrel{\phantom{=}} - \int \Delta^2 \varphi_R |u(t)|^2 dx + 4 \sum_{j,k=1}^N \int_{R \leq |x| \leq 2R} \partial^2_{jk} \varphi_R \rea \left( \partial_j \overline{u}(t) \partial_k u(t) \right) dx \\
		&\mathrel{\phantom{=}} + \frac{2\alpha}{\alpha+2} \int_{R \leq |x| \leq 2R} |x|^{-b} \Delta \varphi_R |u(t)|^{\alpha+2} dx + \frac{4b}{\alpha+2} \int_{R \leq |x| \leq 2R} |x|^{-b-2} x \cdot \nabla \varphi_R |u(t)|^{\alpha+2} dx.
		\end{align*}
		Estimating as above, we get
		\begin{align*}
		\frac{d}{dt} M_{\varphi_R}(t) &\geq 8 \left(\int_{|x| \leq R} |\nabla u(t)|^2 dx + \frac{N\alpha+2b}{2(\alpha+2)} \int_{|x| \leq R} |x|^{-b} |u(t)|^{\alpha+2} dx \right) + O \left( R^{-2} + R^{-\frac{(N-1)\alpha}{2} -b}\right) \\
		&\geq \frac{N\alpha+2b}{2(\alpha+2)} \int_{|x| \leq R} |x|^{-b} |u(t)|^{\alpha+2} dx + O \left( R^{-2} + R^{-\frac{(N-1)\alpha}{2} -b}\right).
		\end{align*}
		Using the fact $0 \leq \chi_R \leq 1$,
		\[
		\int |x|^{-b}|\chi_R u(t)|^{\alpha+2} dx = \int_{|x| \leq R} |x|^{-b} |u(t)|^{\alpha+2} dx - \int_{R/2 \leq |x| \leq R} (1-\chi_R^{\alpha+2}) |x|^{-b} |u(t)|^{\alpha+2} dx,
		\]
		and $\|\Delta(\chi_R)\|_{L^\infty} \lesssim R^{-2}$, the radial Sobolev embedding implies
		\[
		\frac{d}{dt} M_{\varphi_R}(t) \geq \frac{N\alpha+2b}{2(\alpha+2)} \int_{|x| \leq R} |x|^{-b} |\chi_R u(t)|^{\alpha+2} dx + O \left( R^{-2} + R^{-\frac{(N-1)\alpha}{2} -b}\right).
		\]
		Repeating the same reasoning as in the proof of Proposition $\ref{prop-mora-est}$, we complete the proof.
	\end{proof}

\section{Energy scattering in two dimensions}
\label{S3}
\setcounter{equation}{0}
In this section, we give the proof of the energy scattering in two dimensions given in Theorem $\ref{theo-scat-2D-focus}$ and Theorem $\ref{theo-scat-2D-defocus}$. Let us start with the following nonlinear estimates.
\begin{lemma} \label{lem-non-est}
	Let $N=2$, $0<b<1$, $\alpha>2-b$, $0 \leq \gamma \leq 1$ and $I\subset \R$. Then there exists $\theta \in (0,1)$ satisfying $\alpha \theta>1$ such that
	\begin{align} \label{non-est}
	\||\nabla|^\gamma(|x|^{-b} |u|^\alpha u)\|_{S'(L^2,I)} \lesssim \|u\|^{\alpha\theta}_{L^{\alpha+2+b}(I\times \R^2)} \|u\|^{\alpha(1-\theta)}_{L^\infty(I, L^{\infty-})} \||\nabla|^\gamma u\|_{L^\infty(I,L^2)},
	\end{align}
	where $\infty-:=\frac{1}{\epsilon}$ for some $0<\epsilon \ll 1$. In particular,
	\begin{align} \label{non-est-sob}
	\||\nabla|^\gamma(|x|^{-b} |u|^\alpha u)\|_{S'(L^2,I)} \lesssim \|u\|^{\alpha\theta}_{L^{\alpha+2+b}(I\times \R^2)} \|\scal{\nabla} u\|^{\alpha(1-\theta)}_{L^\infty(I, L^2)} \||\nabla|^\gamma u\|_{L^\infty(I,L^2)}.
	\end{align}
\end{lemma}
	
\begin{proof}
	We bound
	\[
	\||\nabla|^\gamma(|x|^{-b} |u|^\alpha u)\|_{S'(L^2,I)} \lesssim \||x|^{-b} |\nabla|^\gamma(|u|^\alpha u)\|_{S'(L^2,I)} + \||x|^{-b-\gamma} |u|^\alpha u\|_{S'(L^2,I)},
	\]
	where we have used the fact $|\nabla|^\gamma(|x|^{-b}) = C(\gamma) |x|^{-b-\gamma}$. Let us first estimate
	\[
	\||x|^{-b} |\nabla|^\gamma(|u|^\alpha u)\|_{S'(L^2,I)} \leq \||x|^{-b} |\nabla|^\gamma(|u|^\alpha u)\|_{L^{m'}(I,L^{n'}(B))} + \||x|^{-b} |\nabla|^\gamma(|u|^\alpha u)\|_{L^{m'}(I,L^{n'}(B^c))}
	\]
	for some Schr\"odinger admissible pair $(m,n)$, where $B$ and $B^c$ are the unit ball and its complement in $\R^2$ respectively. We estimate
	\begin{align} \label{non-est-B}
	\begin{aligned}
	\||x|^{-b} |\nabla|^\gamma (|u|^\alpha u)\|_{L^{m'}(I, L^{n'}(B))} &\leq \||x|^{-b}\|_{L^{\nu}(B)} \||\nabla|^\gamma(|u|^\alpha u)\|_{L^{m'}(I, L^\rho)} \\
	&\lesssim \|u\|^\alpha_{L^{\alpha q}(I, L^{\alpha r})} \||\nabla|^\gamma u\|_{L^\infty(I, L^2)} \\
	&\lesssim \|u\|^{\alpha \theta}_{L^{\alpha+2+b}(I\times \R^2)} \|u\|^{\alpha(1-\theta)}_{L^\infty(I, L^{\infty-})} \||\nabla|^\gamma u\|_{L^\infty(I, L^2)}
	\end{aligned}
	\end{align}
	provided that $\nu, \rho, q, r\geq 1$ satisfy
	\begin{align*}
	\frac{1}{n'} &=\frac{1}{\nu}+\frac{1}{\rho}, & \frac{1}{m'} &= \frac{1}{q} +\frac{1}{\infty}, & \frac{1}{\rho} &= \frac{1}{r}+\frac{1}{2} \\
	\frac{2}{\nu} &>b, & \frac{1}{\alpha q} &= \frac{\theta}{\alpha+2+b} +\frac{1-\theta}{\infty}, & \frac{1}{\alpha r} &=\frac{\theta}{\alpha+2+b} +\frac{1-\theta}{\infty-}.
	\end{align*}
	Note that the condition $\frac{2}{\nu} >b$ ensures $\||x|^{-b}\|_{L^{\nu}} <\infty$. It follows that
	\[
	\frac{1}{m'} = \frac{1}{q}, \quad \frac{1}{n'} = \frac{1}{\nu}+\frac{1}{r}+\frac{1}{2}.
	\]
	Since $\frac{2}{m'}+\frac{2}{n'}=3$, we infer that
	\[
	\frac{2}{\nu} +\frac{4\alpha \theta}{\alpha+2+b} + \frac{2\alpha(1-\theta)}{\infty-} =2.
	\]
	We next take $\frac{2}{\nu}=b+\eta$ for some $\eta>0$ to be chosen shortly and $\infty- = \frac{1}{\epsilon}$ with $0<\epsilon \ll 1$ to get
	\[
	b+\eta + \frac{4\alpha\theta}{\alpha+2+b} + 2\alpha(1-\theta)\epsilon =2
	\]
	or
	\[
	\theta=\theta(\epsilon)= \frac{1}{\frac{4\alpha}{\alpha+2+b} - 2\alpha \epsilon} \left( 2 - b-\eta - 2\alpha \epsilon \right).
	\]
	We see that
	\[
	\theta_0:=\lim_{\epsilon \rightarrow 0} \theta(\epsilon) = \frac{\alpha+2+b}{4\alpha} (2-b-\eta).
	\]
	Since $\alpha>2-b$, by taking $\eta>0$ sufficiently small, it is easy to check that $\theta_0 \in (0,1)$. Moreover, since $0<b<1$ and $\alpha>2-b$, we have that
	\[
	\alpha \theta_0 = \frac{\alpha+2+b}{4}(2-b-\eta)>1
	\]
	provided $\eta>0$ is chosen small enough. Therefore, the estimates in \eqref{non-est-B} are available with some $\theta \in (0,1)$ satisfying $\alpha \theta>1$ by taking $\epsilon, \eta>0$ sufficiently small. The term on $B^c$ is treated similarly by replacing the condition $\frac{2}{\nu}>b$ by $\frac{2}{\nu}<b$. In this case, we just take $\frac{2}{\nu} =b-\eta$ for some $\eta>0$ small enough.
	
	We next estimate
	\[
	\||x|^{-b-\gamma} |u|^\alpha u\|_{S'(L^2,I)} \leq \||x|^{-b-\gamma} |u|^\alpha u\|_{L^{m'}(I, L^{n'}(B))} + \||x|^{-b-\gamma} |u|^\alpha u\|_{L^{m'}(I, L^{n'}(B^c))}
	\]
	for some $(m,n) \in S$. By H\"older's inequality,
	\begin{align*}
	\||x|^{-b-\gamma} |u|^\alpha u\|_{L^{m'}(I,L^{n'}(B))} &\leq \||x|^{-b-\gamma} \|_{L^{\nu}(B)} \||u|^\alpha u\|_{L^{m'}(I,L^\rho)} \\
	&\lesssim \|u\|^\alpha_{L^{\alpha q}(I, L^{\alpha r})} \|u\|_{L^\infty(I, L^\kappa)} \\
	&\lesssim \|u\|^{\alpha \theta}_{L^{\alpha+2+b}(I\times \R^2)} \|u\|^{\alpha(1-\theta)}_{L^\infty(I, L^{\infty-})} \||\nabla|^\gamma u\|_{L^\infty(I, L^2)}
	\end{align*}
	provided that $\nu, \rho, q, r, \kappa\geq 1$ satisfy
	\begin{align*}
	\frac{1}{n'} &=\frac{1}{\nu} + \frac{1}{\rho}, & \frac{1}{m'} &=\frac{1}{q}+\frac{1}{\infty}, &\frac{1}{\rho} &=\frac{1}{r}+\frac{1}{\kappa}, \\
	\frac{2}{\nu} &>b+\gamma, & \frac{1}{\alpha q} &=\frac{\theta}{\alpha+2+b} +\frac{1-\theta}{\infty}, &\frac{1}{\alpha r} &=\frac{\theta}{\alpha+2+b} +\frac{1-\theta}{\infty-}
	\end{align*}
	and $\frac{1}{\kappa} =\frac{1}{2} -\frac{\gamma}{2}$ which comes from the homogeneous Sobolev embedding. Since $\frac{2}{m'}+\frac{2}{n'}=3$, we see that
	\begin{align*}
	\frac{2}{\nu}-\gamma +\frac{4\alpha\theta}{\alpha+2+b} + \frac{2\alpha(1-\theta)}{\infty-}=2.
	\end{align*}
	We take $\frac{2}{\nu}= b+\gamma+\eta$ for some $\eta>0$ to be chosen shortly and $\infty-=\frac{1}{\epsilon}$ with $0<\epsilon \ll 1$ to get
	\[
	b-\eta+\frac{4\alpha\theta}{\alpha+2+b} + 2\alpha(1-\theta)\epsilon =2.
	\]
	It follows that
	\[
	\theta = \theta(\epsilon) = \frac{1}{\frac{4\alpha}{\alpha+2+b}-2\alpha \epsilon} \left(2 -b-\eta-2\alpha \epsilon\right).
	\]
	By the same argument as above, we prove \eqref{non-est}. The estimate \eqref{non-est-sob} follows from \eqref{non-est} and Sobolev embeddings.
\end{proof}

\begin{proposition} \label{prop-glo-Lp-bou}
	Let $N=2$, $0<b<1$ and $\alpha>2-b$. Let $u_0 \in H^1$ be radially symmetric and satisfy \eqref{con-energy} and \eqref{con-grad}. Then the corresponding global solution to the focusing problem \eqref{INLS} satisfies
	\begin{align} \label{glo-Lp-bou}
	\|u\|_{L^{\alpha+2+b}(\R \times \R^2)} \leq C(u_0, Q)<\infty.
	\end{align}
\end{proposition}

\begin{proof}
	Let $\vareps>0$ be a small parameter to be chosen later. By the Sobolev embedding and Strichartz estimates,
	\[
	\|e^{it\Delta} u_0\|_{L^{\alpha+2+b}(\R \times \R^2)} \lesssim \||\nabla|^{\gamma_b} e^{it\Delta} u_0\|_{L^{\alpha+2+b}(\R, L^{\frac{2(\alpha+2+b)}{\alpha+b}})} \lesssim \|u_0\|_{H^1},
	\]
	where $\gamma_b:= \frac{\alpha-2+b}{\alpha+2+b} \in (0,1)$ since $\alpha>2-b$. We split $\R$ into $K=K(\vareps, u_0, Q)$ disjoint intervals $I_k$ such that
	\begin{align} \label{Lp-bou-proo-1}
	\|e^{it\Delta} u_0\|_{L^{\alpha+2+b}(I_k \times \R^2)} <\vareps, \quad \forall k=1, \cdots, K.
	\end{align}
	Let $T$ be a large parameter depending on $\vareps, u_0$ and $Q$. We will prove that
	\begin{align} \label{Lp-bou-proo-2}
	\|u\|_{L^{\alpha+2+b}(I_k \times \R^2)} \lesssim T, \quad \forall k=1, \cdots, K.
	\end{align}
	By summing over all intervals $I_k, k=1, \cdots, K$, we obtain \eqref{glo-Lp-bou}. Let us now prove \eqref{Lp-bou-proo-2}. By Sobolev embedding and the fact $\|u(t)\|_{H^1} \leq C(u_0,Q)$,
	\begin{align} \label{sob-app}
	\|u\|_{L^{\alpha+2+b}(I \times \R^2)}^{\alpha+2+b} \leq C(u_0,Q) |I|
	\end{align}
	for any interval $I \subset \R$. It suffices to show \eqref{Lp-bou-proo-2} with $|I_k|>2T$. Let us fix one such interval, say $I=(a,e)$ with $|I|>2T$. We will show that there exists $t_1 \in (a, a+T)$ such that
	\begin{align} \label{Lp-bou-proo-3}
	\left\| \int_0^{t_1} e^{i(t-s)\Delta} |x|^{-b} |u|^\alpha u(s) ds\right\|_{L^{\alpha+2+b}([t_1,+\infty) \times \R^2)} \leq C(u_0,Q) \vareps^{\frac{1}{2}}.
	\end{align}
	Assume \eqref{Lp-bou-proo-3} for the moment, let us prove \eqref{Lp-bou-proo-2}. By the Duhamel formula
	\[
	e^{i(t-t_1)\Delta} u(t_1) = e^{it\Delta} u_0 + i \int_0^{t_1} e^{i(t-s)\Delta} |x|^{-b} |u|^\alpha u(s) ds,
	\]
	\eqref{Lp-bou-proo-1} and \eqref{Lp-bou-proo-3}, we infer that
	\[
	\|e^{i(t-t_1)\Delta} u(t_1)\|_{L^{\alpha+2+b}([t_1,e] \times \R^2)} \leq C(u_0,Q) \vareps^{\frac{1}{2}}.
	\]
	We also have 
	\[
	u(t) = e^{i(t-t_1)\Delta} u(t_1) + i \int_{t_1}^t e^{i(t-s)\Delta} |x|^{-b} |u|^\alpha u(s) ds.
	\]
	By Strichartz estimates and \eqref{non-est-sob}, 
	\begin{align*}
	\Big\| \int_{t_1}^t e^{i(t-s)\Delta} |x|^{-b} |u|^\alpha u(s) ds &\Big\|_{L^{\alpha+2+b}([t_1,e] \times \R^2)} \\
	&\lesssim \Big\| |\nabla|^{\gamma_b} \int_{t_1}^t e^{i(t-s)\Delta} |x|^{-b} |u|^\alpha u(s) ds \Big\|_{L^{\alpha+2+b}([t_1,e], L^{\frac{2(\alpha+2+b)}{\alpha+b}})} \\
	&\lesssim \||\nabla|^{\gamma_b}(|x|^{-b} |u|^\alpha u)\|_{S'(L^2,[t_1,e])} \\
	&\lesssim \|u\|^{\alpha \theta}_{L^{\alpha+2+b}([t_1,e]\times \R^2)} \|\scal{\nabla} u\|^{\alpha(1-\theta)}_{L^\infty([t_1,e], L^2)} \||\nabla|^{\gamma_b} u\|_{L^\infty([t_1,e], L^2)} \\
	&\leq C(u_0,Q) \|u\|^{\alpha\theta}_{L^{\alpha+2+b}([t_1,e]\times \R^2)}
	\end{align*}
	for some $\theta \in (0,1)$ satisfying $\alpha\theta>1$. It follows that
	\begin{align*}
	\|u\|_{L^{\alpha+2+b}([t_1,e]\times \R^2)} &\leq \|e^{i(t-t_1)\Delta} u_0\|_{L^{\alpha+2+b}([t_1,e]\times \R^2)} + C(u_0,Q)\|u\|^{\alpha\theta}_{L^{\alpha+2+b}([t_1,e] \times \R^2)} \\
	&\leq C(u_0,Q) \vareps^{\frac{1}{2}} + C(u_0,Q) \|u\|^{\alpha\theta}_{L^{\alpha+2+b}([t_1,e]\times \R^2)}
	\end{align*}
	which, by the continuity argument, implies that
	\begin{align} \label{est-t1-e}
	\|u\|_{L^{\alpha+2+b}([t_1,e]\times \R^2)} \leq C(\vareps, u_0,Q).
	\end{align}
	On the other hand, by \eqref{sob-app} and the fact $t_1-a<T$, we see that
	\begin{align} \label{est-a-t1}
	\|u\|_{L^{\alpha+2+b}([a,t_1]\times \R^2)} \leq C(u_0,Q) |t_1-a|^{\frac{1}{\alpha+2+b}} \leq C(u_0,Q) T^{\frac{1}{\alpha+2+b}}.
	\end{align}
	Combining \eqref{est-t1-e} and \eqref{est-a-t1}, we prove \eqref{Lp-bou-proo-2}. 
	
	It remains to show \eqref{Lp-bou-proo-3}. By the time translation, we may assume that $a=0$. We first claim that there exists $t_0 \in \left[\frac{T}{4}, \frac{T}{2}\right]$ such that
	\begin{align} \label{Lp-bou-proo-4}
	\int_{t_0}^{t_0+\vareps T^{1-\beta_2}} \|u(s)\|^{\alpha+2+b}_{L^{\alpha+2+b}} ds \leq C(u_0,Q) \vareps,
	\end{align}
	where $\beta_2$ is as in \eqref{mora-est-I}. Indeed, we cover the interval $J=\left[\frac{T}{4}, \frac{T}{2}\right]$ by $L=\vareps^{-1} T^{\beta_2}$ disjoint intervals $J_l$ of length $\vareps T^{1-\beta_2}$ and use \eqref{mora-est-I} to have
	\[
	L \min_{1\leq l \leq L} \int_{J_l} \|u(s)\|^{\alpha+2+b}_{L^{\alpha+2+b}} ds \leq \sum_{l=1}^L \int_{J_l} \|u(s)\|^{\alpha+2+b}_{L^{\alpha+2+b}} ds = \int_J \|u(s)\|^{\alpha+2+b}_{L^{\alpha+2+b}} ds \leq C(u_0,Q) T^{\beta_2}.
	\]
	There thus exists $l_0 \in \{1, \cdots, L\}$ such that
	\[
	\int_{I_{l_0}} \|u(s)\|^{\alpha+2+b}_{L^{\alpha+2+b}} ds \leq C(u_0,Q) \vareps
	\]
	which proves the claim. We now set
	\begin{align} \label{def-t1}
	t_1:= t_0 + \vareps T^{1-\beta_2}.
	\end{align}
	Since $t_0<\frac{T}{2}$, by reducing $\vareps$ if necessary, we may assume that $t_1<T$. We will estimate the time interval in \eqref{Lp-bou-proo-3} by considering separately $[0,t_0]$ and $[t_0,t_1]$. On $[0,t_0]$, we use the dispersive estimate to get
	\[
	\left\| \int_0^{t_0} e^{i(t-s)\Delta} |x|^{-b} |u|^\alpha u(s) ds \right\|_{L^\infty} \lesssim \int_0^{t_0} |t-s|^{-1} \left( \int |x|^{-b} |u(s,x)|^{\alpha+1} dx\right) ds.
	\]
	By H\"older's inequality, we estimate
	\begin{align*}
	\int |x|^{-b} |u(s,x)|^{\alpha+1} dx &\leq \||x|^{-\frac{b}{2}} |u(s)|^{\frac{\alpha+2}{2}} \|_{L^2} \||x|^{-\frac{b}{2}}\|_{L^{\nu_1}(B)} \||u(s)|^{\frac{\alpha}{2}}\|_{L^{\rho_1}}  \\
	&\mathrel{\phantom{\leq}}+ \||x|^{-\frac{b}{2}} |u(s)|^{\frac{\alpha+2}{2}} \|_{L^2} \||x|^{-\frac{b}{2}}\|_{L^{\nu_2}(B)} \||u(s)|^{\frac{\alpha}{2}}\|_{L^{\rho_2}} \\
	&\leq \||x|^{-\frac{b}{2}} |u(s)|^{\frac{\alpha+2}{2}} \|_{L^2} \||x|^{-\frac{b}{2}}\|_{L^{\nu_1}(B)} \|u(s)\|^{\frac{\alpha}{2}}_{L^{^{\frac{\alpha \rho_1}{2}}}} \\
	&\mathrel{\phantom{\leq}}+ \||x|^{-\frac{b}{2}} |u(s)|^{\frac{\alpha+2}{2}} \|_{L^2} \||x|^{-\frac{b}{2}}\|_{L^{\nu_2}(B)} \|u(s)\|^{\frac{\alpha}{2}}_{L^{^{\frac{\alpha \rho_2}{2}}}} \\
	&\lesssim \left(\int |x|^{-b} |u(s,x)|^{\alpha+2} dx\right)^{\frac{1}{2}}
	\end{align*}
	provided that $\nu_1, \nu_2, \rho_1, \rho_2 \geq 1$ satisfy
	\begin{align*}
	\frac{1}{2}=\frac{1}{\nu_1} + \frac{1}{\rho_1} =\frac{1}{\nu_2} +\frac{1}{\rho_2}, \quad \frac{2}{\nu_1}>\frac{b}{2}, \quad \frac{2}{\nu_2}<\frac{b}{2}, \quad \frac{\alpha \rho_1}{2} >2, \quad \frac{\alpha \rho_2}{2}>2.
	\end{align*}
	Since $0<b<1$ and $\alpha>2-b$, it is easy to check that the above conditions hold for a suitable choice of $\nu_1,\nu_2, \rho_1$ and $\rho_2$. 
	Thanks to \eqref{mora-est}, we have for $t>t_1$ 
	\begin{align*}
	\left\| \int_0^{t_0} e^{i(t-s)\Delta} |x|^{-b} |u|^\alpha u(s) ds \right\|_{L^\infty} &\lesssim \int_0^{t_0} |t-s|^{-1} \left(\int |x|^{-b} |u(s,x)|^{\alpha+2} dx\right)^{\frac{1}{2}} ds \\
	&\lesssim \left( \int_0^{t_0} |t-s|^{-2} ds\right)^{\frac{1}{2}} \left( \int_0^{t_0} \int |x|^{-b} |u(s,x)|^{\alpha+2} dx ds \right)^{\frac{1}{2}} \\
	&\lesssim |t-t_0|^{-\frac{1}{2}} |t_0|^{\frac{\beta_1}{2}} \\
	&\lesssim |t_1-t_0|^{-\frac{1}{2}} |t_0|^{\frac{\beta_1}{2}} \\
	&\lesssim (\vareps T^{1-\beta_2})^{-\frac{1}{2}} T^{\frac{\beta_1}{2}} \\
	&\lesssim (\vareps T^{1-\beta_1-\beta_2})^{-\frac{1}{2}}.
	\end{align*}
	Note that since $0<b<1$ and $\alpha>2-b$, it is easy to see that $\beta_1+\beta_2 <1$. 
	We thus obtain
	\[
	\left\| \int_0^{t_0} e^{i(t-s)\Delta} |x|^{-b} |u|^\alpha u(s) ds \right\|_{L^\infty([t_1,+\infty) \times \R^2)} \leq C(u_0,Q) (\vareps T^{1-\beta_1-\beta_2})^{-\frac{1}{2}}.
	\]
	On the other hand, we use the fact
	\[
	i \int_0^{t_0} e^{i(t-s)\Delta} |x|^{-b} |u|^\alpha u(s) ds = e^{i(t-t_0)\Delta} u_0 - e^{it\Delta} u_0
	\]
	and Strichartz estimates to have
	\[
	\left\| \int_0^{t_0} e^{i(t-s)\Delta} |x|^{-b} |u|^\alpha u(s) ds \right\|_{L^4([t_1,+\infty) \times \R^2)} \leq C(u_0,Q).
	\]
	Interpolating between $L^4$ and $L^\infty$, it yields
	\begin{align*}
	\left\| \int_0^{t_0} e^{i(t-s)\Delta} |x|^{-b} |u|^\alpha u(s) ds\right\|_{L^{\alpha+2+b}([t_1,+\infty) \times \R^2)} \leq C(u_0,Q) \left(\vareps T^{1-\beta_1-\beta_2} \right)^{-\frac{\alpha-2+b}{2(\alpha+2+b)}}.
	\end{align*}
	On $[t_0,t_1]$, we use \eqref{Lp-bou-proo-4} and \eqref{non-est-sob} to have
	\begin{align*}
	\Big\|\int_{t_0}^{t_1} e^{i(t-s)\Delta} |x|^{-b} |u|^\alpha u(s) ds &\Big\|_{L^{\alpha+2+b}([t_1,+\infty)\times \R^2)} \\
	&\lesssim \||\nabla|^{\gamma_b}(|x|^{-b} |u|^\alpha u)\|_{S'(L^2, [t_0,t_1])} \\
	&\lesssim \|u\|^{\alpha \theta}_{L^{\alpha+2+b}([t_0,t_1]\times \R^2)} \|\scal{\nabla} u\|^{\alpha(1-\theta)}_{L^\infty([t_0,t_1],L^2)} \||\nabla|^{\gamma_b} u\|_{L^\infty([t_0,t_1], L^2)} \\
	&\leq C(u_0,Q) \vareps^{\frac{\alpha\theta}{\alpha+2+b}}.
	\end{align*}
	We thus obtain
	\[
	\left\|\int_0^{t_1} e^{i(t-s)\Delta} |x|^{-b} |u|^\alpha u(s) ds \right\|_{L^{\alpha+2+b}([t_1,+\infty)\times \R^2)} \leq C(u_0,Q) \left[\left(\vareps T^{1-\beta_1-\beta_2} \right)^{-\frac{\alpha-2+b}{2(\alpha+2+b)}} + \vareps^{\frac{\alpha\theta}{\alpha+2+b}} \right].
	\]
	Note that by taking $\infty- = \frac{1}{\epsilon}$ in Lemma $\ref{lem-non-est}$ with $\epsilon>0$ sufficiently small, we see that $\frac{\alpha \theta}{\alpha+2+b} >\frac{1}{2}$. By taking $T$ large enough depending on $\vareps$, we prove \eqref{Lp-bou-proo-3}. The proof is complete.
	\end{proof}

\begin{corollary} \label{coro-glo-Lp-bou}
	Let $N=2$, $0<b<1$ and $\alpha>2-b$. Let $u_0 \in H^1$ be radially symmetric. Then the corresponding global solution to the defocusing problem \eqref{INLS} satisfies
	\[
	\|u\|_{L^{\alpha+2+b}(\R \times \R^2)} \leq C(u_0)<\infty.
	\]
\end{corollary}

\begin{proof}
	The proof is similar to the one of Proposition $\ref{prop-glo-Lp-bou}$ by using \eqref{mora-est-I-defocus} instead of \eqref{mora-est-I}.
\end{proof}

We are now able to show the energy scattering given in Theorem $\ref{theo-scat-2D-focus}$.

\noindent {\bf Proof of Theorem $\ref{theo-scat-2D-focus}$.}
We first show that the global bound \eqref{glo-Lp-bou} implies the global Strichartz bound
\[
\|\scal{\nabla} u\|_{S(L^2,\R)} \leq C(u_0,Q)<\infty.
\]
To see this, we use Strichartz estimates, \eqref{non-est-sob} and \eqref{glo-Lp-bou} to have
\begin{align*}
\|\scal{\nabla} u\|_{S(L^2,\R)} &\leq \|\scal{\nabla}u_0\|_{L^2} + \|\scal{\nabla}(|x|^{-b} |u|^\alpha u)\|_{S'(L^2, \R)} \\
&\leq C(u_0,Q) + \|u\|^{\alpha \theta}_{L^{\alpha+2+b}(\R\times \R^2)} \|\scal{\nabla} u\|^{1+\alpha(1-\theta)}_{L^\infty(\R,L^2)} \\
&\leq C(u_0,Q).
\end{align*}
We now show the energy scattering of the global solution. By the time reversal symmetry, it suffices to consider positive times. By Duhamel formula, Strichartz estimates and \eqref{non-est-sob}, we have for $0<t_1<t_2$,
\begin{align*}
\|e^{-it_2\Delta} u(t_2) - e^{-it_1\Delta} u(t_1)\|_{H^1} & = \left\| i \int_{t_1}^{t_2} e^{-is\Delta} |x|^{-b} |u|^\alpha u(s) ds \right\|_{H^1} \\
&\lesssim \|\scal{\nabla} (|x|^{-b} |u|^\alpha u)\|_{S'(L^2,[t_1,t_2])} \\
&\lesssim \|u\|^{\alpha\theta}_{L^{\alpha+2+b}([t_1,t_2]\times \R^2)} \|\scal{\nabla} u\|^{1+\alpha(1-\theta)}_{L^\infty([t_1,t_2], L^2)}.
\end{align*}
Thanks to \eqref{glo-Lp-bou}, we see that
\[
\|e^{-it_2\Delta} u(t_2) - e^{-it_1\Delta} u(t_1)\|_{H^1} \rightarrow 0 \text{ as } t_1, t_2 \rightarrow +\infty.
\]
Thus the limit
\[
u^+_0:= \lim_{t\rightarrow +\infty} e^{-it\Delta} u(t) = u_0 + i\int_0^{+\infty} e^{-is\Delta} |x|^{-b} |u|^\alpha u(s) ds
\]
exists in $H^1$. Arguing as above, we prove as well that
\[
\|u(t)-e^{it\Delta} u_0^+\|_{H^1} \rightarrow 0 \text{ as } t\rightarrow +\infty.
\]
The proof is complete.
\hfill $\Box$

\noindent {\bf Proof of Theorem $\ref{theo-scat-2D-defocus}$.}
The proof is completely similar to the one of Theorem $\ref{theo-scat-2D-focus}$ using Corollary $\ref{coro-glo-Lp-bou}$.
\hfill $\Box$

	\section*{Acknowledgement}
	This work was supported in part by the Labex CEMPI (ANR-11-LABX-0007-01). The author would like to express his deep gratitude to his wife - Uyen Cong for her encouragement and support. He also would like to thank the reviewer for his/her helpful comments and suggestions. 
	
	\appendix
	
	\section{Alternative proof for the energy scattering in dimensions $N\geq 3$}
	\label{S4}
	\setcounter{equation}{0}
	In this section, we give an alternative proof for the energy scattering of non-radial solutions to the defocusing problem \eqref{INLS} in dimensions $N\geq 3$.
	
	\begin{theorem} \label{theo-ene-sca-INLS}
		Let 
		\[
		N\geq 4, \quad 0<b<2, \quad \frac{4-2b}{N} < \alpha <\frac{4-2b}{N-2},
		\]
		and
		\[
		N=3, \quad 0<b<\frac{5}{4}, \quad \frac{4-2b}{3} <\alpha <3-2b.
		\]
		Let $u_0 \in H^1$ and $u$ be the corresponding global solution to the defocusing problem \eqref{INLS}. Then there exist $u_0^\pm \in H^1$ such that 
		\[
		\lim_{t\rightarrow \pm \infty} \|u(t) -e^{it\Delta} u_0^\pm\|_{H^1} =0.
		\]
	\end{theorem}
	This result has been obtained in \cite{Dinh-scat} by using the decaying property of global solutions. Here we present a shorter proof via the interaction Morawetz inequality.
	
	We have from \cite[Proposition 4.7]{Dinh-repul} (by taking $V=0$ and $W=|x|^{-b}$) that the following interaction Morawetz inequality holds true for the defocusing problem \eqref{INLS} in dimensions $N\geq 3$
	\begin{align} \label{inter-morawetz}
	\||\nabla|^{-\frac{N-3}{4}} u\|_{L^4(\R,L^4)} \leq \|u\|_{L^\infty(\R,L^2)}^{\frac{3}{4}} \|\nabla u\|_{L^\infty(\R,L^2)}^{\frac{1}{4}}.
	\end{align}
	
	Using \eqref{inter-morawetz}, the interpolation inequality yields
	\[
	\|u\|_{L^{\frac{N-3+4\sigma}{\sigma}}(\R, L^{\frac{2(N-3+4\sigma)}{N-3+2\sigma}})} \lesssim \||\nabla|^{-\frac{N-3}{4}} u\|_{L^4(\R,L^4)}^{\frac{4\sigma}{N-3+4\sigma}} \||\nabla|^\sigma u\|_{L^\infty(\R,L^2)}^{\frac{N-3}{N-3+4\sigma}},
	\]
	for $0 \leq \sigma \leq 1$. Taking $\sigma=1$, we get
	\[
	\|u\|_{L^{N+1}(\R, L^{\frac{2(N+1)}{N-1}})} \lesssim \left(\|u\|_{L^\infty(J,L^2)}^{\frac{3}{4}} \|\nabla u\|_{L^\infty(\R,L^2)}^{\frac{1}{4}} \right)^{\frac{4}{N+1}} \|\nabla u\|_{L^\infty(\R,L^2)}^{\frac{N-3}{N+1}} = \|u\|_{L^\infty(\R,L^2)}^{\frac{3}{N+1}} \|\nabla u\|_{L^\infty(\R,L^2)}^{\frac{N-2}{N+1}}.
	\]
	By the conservation of mass and energy, we obtain the global bound for global solutions to defocusing problem \eqref{INLS} in dimensions $N\geq 3$,
	\begin{align} \label{glo-bou}
	\|u\|_{L^{N+1}(\R, L^{\frac{2(N+1)}{N-1}})} \leq C(E,M)<\infty. 
	\end{align}
	
	To show the energy scattering, we need the following nonlinear estimates. 
	\begin{lemma} \label{lem-non-est-defo}
		Let $N, b$ and $\alpha$ be as in Theorem $\ref{theo-ene-sca-INLS}$. Let $u$ be the global solution to the defocusing problem \eqref{INLS}. Then there exists $\eps>0$ small enough such that for any time interval $J$ and $k=0, 1$,
		\begin{align*}
		\||x|^{-b} |\nabla|^k(|u|^\alpha u)\|_{L^2(J, L^{\frac{2N}{N+2}})} &\lesssim \||\nabla|^k u\|_{L^{2+\eps}(J, L^{\frac{2N(2+\eps)}{N(2+\eps)-4}})} \|u\|^{\frac{\eps(N+1)}{2(2+\eps)}}_{L^{N+1}(J, L^{\frac{2(N+1)}{N-1}})} \|u\|_{L^\infty(J, L^2)}^{a_1(\eps)} \|\nabla u\|_{L^\infty(J,L^2)}^{b_1(\epsilon)}, 
		\end{align*}
		and
		\begin{align*}
		\||x|^{-b-1} |u|^\alpha u\|_{L^2(J, L^{\frac{2N}{N+2}})} &\lesssim \|\scal{\nabla} u\|_{L^{2+\eps}(J, L^{\frac{2N(2+\eps)}{N(2+\eps)-4}})} \|u\|^{\frac{\eps(N+1)}{2(2+\eps)}}_{L^{N+1}(J, L^{\frac{2(N+1)}{N-1}})} \|u\|_{L^\infty(J, L^2)}^{a_2(\eps)} \|\nabla u\|_{L^\infty(J,L^2)}^{b_2(\epsilon)},
		\end{align*}
		for some positive numbers $a_1(\eps), b_1(\eps), a_2(\eps)$ and $b_2(\eps)$.
	\end{lemma}
	
	\begin{proof}
		We write
		\[
		\||x|^{-b} |\nabla|^k(|u|^\alpha u)\|_{L^2(J,L^{\frac{2N}{N+2}})} \leq \||x|^{-b} |\nabla|^k(|u|^\alpha u)\|_{L^2(J,L^{\frac{2N}{N+2}}(B))} + \||x|^{-b} |\nabla|^k(|u|^\alpha u)\|_{L^2(J,L^{\frac{2N}{N+2}}(B^c))}.
		\]
		By H\"older's inequality and the fractional chain rule,
		\begin{align*}
		\||x|^{-b} |\nabla|^k(|u|^\alpha u)\|_{L^2(J,L^{\frac{2N}{N+2}}(B))} &\leq \||x|^{-b}\|_{L^\gamma(B)} \||\nabla|^k(|u|^\alpha u)\|_{L^2(J, L^m)} \\
		&\lesssim \||\nabla|^k u\|_{L^{2+\eps}(J, L^{\frac{2N(2+\eps)}{N(2+\eps)-4}})} \|u\|^\alpha_{L^{\frac{2\alpha(2+\eps)}{\eps}}(J, L^n)},
		\end{align*}
		provided that $\gamma, m, n\geq 1$ satisfy
		\[
		\frac{N}{\gamma}>b, \quad \frac{N+2}{2N} = \frac{1}{\gamma}+\frac{1}{m}, \quad \frac{1}{m} = \frac{N(2+\eps)-4}{2N(2+\eps)} + \frac{\alpha}{n}.
		\]
		Similar estimates hold on $B^c$ provided that the above conditions are satisfied with $\frac{N}{\gamma}<b$ instead of $\frac{N}{\gamma}>b$. We next bound
		\[
		\|u\|_{L^{\frac{2\alpha(2+\eps)}{\eps}}(J, L^n)} \leq \|u\|^{\theta_1}_{L^{N+1}(J,L^{\frac{2(N+1)}{N-1}})} \|u\|^{1-\theta_1}_{L^\infty(J, L^q)},
		\]
		provided that $\theta_1= \frac{\eps(N+1)}{2\alpha(2+\eps)}$ and $q\geq 1$ satisfies
		\[
		\frac{1}{n}=\frac{(N-1)\theta_1}{2(N+1)} + \frac{1-\theta_1}{q}.
		\]
		We continue to bound
		\[
		\|u\|_{L^\infty(J, L^q)} \leq \|u\|^{\theta_2}_{L^\infty(J,L^2)} \|u\|^{1-\theta_2}_{L^\infty(J,L^{\frac{2N}{N-2}})} \lesssim \|u\|^{\theta_2}_{L^\infty(J,L^2)} \|\nabla u\|^{1-\theta_2}_{L^\infty(J,L^2)},
		\]
		provided that $\frac{1}{q}= \frac{\theta_2}{2} + \frac{(1-\theta_2)(N-2)}{2N}$. We thus obtain 
		\[
		\|u\|^\alpha_{L^{\frac{2\alpha(2+\eps)}{\eps}}(J, L^n)} \lesssim \|u\|^{\frac{\eps(N+1)}{2(2+\eps)}}_{L^{N+1}(J,L^{\frac{2(N+1)}{N-1}})} \|u\|^{a_1(\eps)}_{L^\infty(J,L^2)} \|\nabla u\|^{b_1(\eps)}_{L^\infty(J,L^2)},
		\]
		where 
		\begin{align*}
		a_1(\eps) &= \alpha(1-\theta_1)\theta_2 = \frac{N+2}{2}-\frac{N}{\gamma} - \frac{2(N-2)(\alpha+1) + \eps \left(N+1+(N-2)\alpha\right) }{2(2+\eps)}, \\
		b_1(\eps) &= \alpha(1-\theta_1) (1-\theta_2) = \frac{N}{\gamma}-\frac{N+2}{2} + \frac{2(N\alpha+N-2) + N\alpha\eps}{2(2+\eps)}.
		\end{align*}
		In order to perform the above estimates, we need $a_1(\eps)>0$ and $b_1(\eps)>0$. Since the functions $\eps \mapsto a_1(\eps)$ and $\eps \mapsto b_1(\eps)$ are decreasing, it suffices to show their limits as $\eps \rightarrow 0$ are positive. We have that
		\[
		\lim_{\eps \rightarrow 0} a_1(\eps) = \frac{N+2}{2} - \frac{N}{\gamma} -\frac{(N-2)(\alpha+1)}{2}, \quad \lim_{\eps \rightarrow 0} b_1(\eps) = \frac{N}{\gamma} - \frac{N+2}{2} + \frac{N\alpha+N-2}{2}.
		\]
		On $B$, we need $\frac{N}{\gamma} >b$. Set $\frac{N}{\gamma} = b+\eta$ for some $\eta>0$ to be chosen shortly. We have that 
		\[
		\lim_{\eps \rightarrow 0} a_1(\eps) = \frac{4-2b-(N-2)\alpha}{2} -\eta, \quad \lim_{\eps \rightarrow 0} b_1(\eps) = \frac{N\alpha - 4+2b}{2} + \eta.
		\]
		Since $\frac{4-2b}{N}<\alpha<\frac{4-2b}{N-2}$, we can choose $0<\eta<\frac{4-2b-(N-2)\alpha}{2}$ so that these two limits are positive. Similarly, on $B^c$, we write $\frac{N}{\gamma} =b -\eta$ for some $\eta>0$ to be chosen later. We see that
		\[
		\lim_{\eps \rightarrow 0} a_1(\eps) = \frac{4-2b-(N-2)\alpha}{2} + \eta, \quad \lim_{\eps \rightarrow 0} b_1(\eps) = \frac{N\alpha - 4+2b}{2} - \eta.
		\]
		Thus by choosing $0<\eta<\frac{N\alpha-4+2b}{2}$, the two limits are positive. 
		
		Let us now show the second estimate. We again write
		\[
		\||x|^{-b-1} |u|^\alpha u\|_{L^2(J, L^{\frac{2N}{N+2}})} \leq \||x|^{-b-1} |u|^\alpha u\|_{L^2(J, L^{\frac{2N}{N+2}}(B))} + \||x|^{-b-1} |u|^\alpha u\|_{L^2(J, L^{\frac{2N}{N+2}}(B^c))}.
		\]
		Let us consider two cases $N\geq 4$ and $N=3$. 
		
		\underline{When $N\geq 4$}, we use H\"older's inequality and Sobolev embedding to have 
		\begin{align*}
		\||x|^{-b-1} |u|^\alpha u\|_{L^2(J, L^{\frac{2N}{N+2}}(B))} &\leq \||x|^{-b-1}\|_{L^\gamma(B)} \||u|^\alpha u\|_{L^2(J, L^m)} \\
		&\lesssim \|u\|_{L^{2+\eps}(J, L^{\frac{2N(2+\eps)}{(N-2)(2+\eps)-4}})} \|u\|^\alpha_{L^{\frac{2\alpha(2+\eps)}{\eps}}(J,L^n)} \\
		&\lesssim \|\nabla u\|_{L^{2+\eps}(J, L^{\frac{2N(2+\eps)}{N(2+\eps)-4}})} \|u\|^\alpha_{L^{\frac{2\alpha(2+\eps)}{\eps}}(J,L^n)},
		\end{align*}
		provided that $\gamma, m, n\geq 1$ satisfy
		\[
		\frac{N}{\gamma}>b+1, \quad \frac{N+2}{2N}=\frac{1}{\gamma}+\frac{1}{m}, \quad \frac{1}{m} = \frac{(N-2)(2+\eps)-4}{2N(2+\eps)} + \frac{\alpha}{n}.
		\]
		We estimate similarly for the term involving $B^c$ provided that the first condition is replaced by $\frac{N}{\gamma}< b+1$. Estimating as above, we get
		\[
		\|u\|^\alpha_{L^{\frac{2\alpha(2+\eps)}{\eps}}(J,L^n)} \lesssim \|u\|^{\frac{\eps(N+1)}{2(2+\eps)}}_{L^{N+1}(J, L^{\frac{2(N+1)}{N-1}})} \|u\|^{a_2(\eps)}_{L^\infty(J, L^2)} \|\nabla u\|^{b_2(\eps)}_{L^\infty(J,L^2)},
		\]
		where
		\begin{align*}
		a_2(\eps) &= \frac{N+2}{2}-\frac{N}{\gamma}+1 - \frac{2(N-2)(\alpha+1) + \eps \left( N+1+(N-2)\alpha\right)}{2(2+\eps)}, \\
		b_2(\eps) &= \frac{N}{\gamma} - \frac{N+2}{2} -1 + \frac{2(N\alpha+N-2) +N\alpha\eps}{2(2+\eps)}.
		\end{align*}
		Since $\eps \mapsto a_2(\eps)$ and $\eps \mapsto b_2(\eps)$ are decreasing, it remains to show 
		\[
		\lim_{\eps \rightarrow 0} a_2(\eps) = \frac{N+2}{2} - \frac{N}{\gamma} +1  -\frac{(N-2)(\alpha+1)}{2} >0, \quad \lim_{\eps \rightarrow 0} b_2(\eps) = \frac{N}{\gamma} - \frac{N+2}{2} -1 + \frac{N\alpha+N-2}{2}>0.
		\]
		On $B$, we take $\frac{N}{\gamma}=b+1+\eta$ for some $0<\eta<\frac{4-2b-(N-2)\alpha}{2}$. It is easy to see that these two limits are positive. Similarly, on $B^c$, we can take $\frac{N}{\gamma}=b+1-\eta$ with some $0<\eta<\frac{N\alpha-4+2b}{2}$ so that the two limits are positive. 
		
		\underline{When $N=3$}, we note that the above argument does not hold since $\frac{2N(2+\eps)}{(N-2)(2+\eps)-4}$ is negative for $\eps>0$ small. We estimate
		\begin{align*}
		\||x|^{-b-1} |u|^\alpha u\|_{L^2(J, L^{\frac{6}{5}}(B))} &\leq \||x|^{-b-1}\|_{L^\gamma(B)} \||u|^\alpha u\|_{L^2(J,L^m)} \\
		&\leq \|u\|_{L^{2+\eps}(J,L^r)} \|u\|^\alpha_{L^{\frac{2\alpha(2+\eps)}{\eps}}(J,L^n)} \\
		&\lesssim \|\scal{\nabla} u\|_{L^{2+\eps}(J,L^{\frac{6(2+\eps)}{3(2+\eps)-4}})} \|u\|^\alpha_{L^{\frac{2\alpha(2+\eps)}{\eps}}(J,L^n)},
		\end{align*}
		provided that $\gamma, m, r, n\geq 1$ satisfy
		\[
		\frac{3}{\gamma} > b+1, \quad \frac{5}{6}=\frac{1}{\gamma} + \frac{1}{m}, \quad \frac{1}{m} = \frac{1}{r}+\frac{\alpha}{n}, \quad \frac{6(2+\eps)}{3(2+\eps)-4} < r<\infty. 
		\]
		Here the last condition ensures the inhomogeneous Sobolev embedding. The same estimates hold on $B^c$ provided that the condition $\frac{3}{\gamma}>b+1$ is replaced by $\frac{3}{\gamma}<b+1$. We can rewrite the last condition as $\frac{1}{r}=\frac{(2+3\eps)\tau}{12+6\eps}$ for some $\tau \in (0,1)$. We estimate as above to get
		\[
		\|u\|^\alpha_{L^{\frac{2\alpha(2+\eps)}{\eps}}(J, L^n)} \lesssim \|u\|^{\frac{2\eps}{2+\eps}}_{L^4(J, L^4)} \|u\|^{a_2(\eps)}_{L^\infty(J,L^2)} \|\nabla u\|^{a_2(\eps)}_{L^\infty(J,L^2)},
		\]
		where
		\begin{align*}
		a_2(\eps) &= \frac{5}{2}-\frac{3}{\gamma} - \frac{2(\alpha+\tau) +\eps\left(3\tau + 3 - (2-\alpha)\right)}{2(2+\eps)}, \\
		b_2(\eps) &= \frac{3}{\gamma} -\frac{5}{2} + \frac{2(3\alpha+\tau) + \eps\left(3\tau +3 - 3(2-\alpha)\right)}{2(2+\eps)}.
		\end{align*}
		It is not hard to check that $\eps \mapsto a_2(\eps)$ and $\eps \mapsto b_2(\eps)$ are decreasing. On the other hand,
		\[
		\lim_{\eps \rightarrow 0} a_2(\eps) = \frac{5}{2}-\frac{3}{\gamma} - \frac{\alpha +\tau}{2}, \quad \lim_{\eps \rightarrow 0} b_2(\eps) = \frac{3}{\gamma} - \frac{5}{2} + \frac{3\alpha+\tau}{2}.
		\]
		Note that the limit $\lim_{\eps \rightarrow 0} a_2(\eps)$ attains its maximum value as $\tau \rightarrow 0$. We thus need to choose $\tau$ close to 0. 
		
		On $B$, we take $\frac{3}{\gamma}=1+b +\eta$ for some $\eta>0$ to be chosen shortly. We see that
		\[
		\lim_{\eps \rightarrow 0} a_2(\eps) = \frac{3-2b-\tau - \alpha}{2} -\eta, \quad \lim_{\eps \rightarrow 0} b_2(\eps) = \frac{3\alpha - 3 +2b+\tau}{2} +\eta.
		\]
		Since $\frac{4-2b}{3} <\alpha$, the second limit is positive for any $\tau \in (0,1)$. By choosing $0<\tau<3-2b-\alpha$ and $0<\eta<\frac{3-2b -\alpha-\tau}{2}$, the first limit is positive provided that $\alpha <3-2b$. This leads to the restriction 
		\[
		\frac{4-2b}{3}<\alpha<3-2b, \quad 0<b<\frac{5}{4}.
		\]
		On $B^c$, we take $\frac{3}{\gamma}=1+b-\eta$ for some $\eta>0$ to be chosen later. By choosing $0<\tau<3-2b-\alpha$ and $0<\eta<\frac{3\alpha-3+2b+\tau}{2}$, the two limits are positive. Taking $\eps>0$ sufficiently small, we prove the result.
	\end{proof}
	
	\noindent \textbf{Proof of Theorem $\ref{theo-ene-sca-INLS}$.}
	We first show that the global Morawetz bound \eqref{glo-bou} implies the global Strichartz bound
	\begin{align} \label{glo-str-bou}
	\|\scal{\nabla}u\|_{S(L^2,\R)} \leq C(E,M)<\infty.
	\end{align}
	To see this, we decompose $\R$ into a finite number of disjoint intervals $J_l = [t_l, t_{l+1}], l = 1, \cdots, L$ so that
	\[
	\|u\|_{L^{N+1}(J_l, L^{\frac{2(N+1)}{N-1}})} \leq \delta, \quad l =1, \cdots, L
	\]
	for some small constant $\delta>0$ to be chosen later. By Strichartz estimates, we have that
	\begin{align*}
	\|\scal{\nabla}u\|_{S(L^2,J_l)} &\lesssim \|\scal{\nabla} u(t_l)\|_{L^2} + \|\scal{\nabla}(|x|^{-b} |u|^\alpha u)\|_{L^2(J_l, L^{\frac{2N}{N+2}})} \\
	&\lesssim \|\scal{\nabla} u(t_l)\|_{L^2} + \sum_{k=0}^1\||x|^{-b} |\nabla|^k(|u|^\alpha u)\|_{L^2(J_l, L^{\frac{2N}{N+2}})} + \||x|^{-b-1} |u|^\alpha u\|_{L^2(J_l,L^{\frac{2N}{N+2}})}.
	\end{align*}
	We learn from Lemma $\ref{lem-non-est-defo}$ that for $\eps>0$ small enough, there exist positive numbers $a_1(\eps), b_1(\eps)$, $a_2(\eps)$ and $b_2(\eps)$ such that
	\begin{align*}
	\|\scal{\nabla} u\|_{S(L^2,J_l)} &\lesssim \|\scal{\nabla} u(t_l)\|_{L^2} + \|\scal{\nabla} u\|_{S(L^2,J_l)} \|u\|^{\frac{\eps(N+1)}{2(2+\eps)}}_{L^{N+1}(J_l, L^{\frac{2N}{N+2}})} \|u\|_{L^\infty(J_l,L^2)}^{a_1(\eps)} \|\nabla u\|_{L^\infty(J_l,L^2)}^{b_1(\eps)} \\
	&\mathrel{\phantom{\lesssim \|\scal{\nabla} u(t_l)\|_{L^2}}} +\|\scal{\nabla}u\|_{S(L^2,J_l)} \|u\|^{\frac{\eps(N+1)}{2(2+\eps)}}_{L^{N+1}(J_l, L^{\frac{2N}{N+2}})} \|u\|_{L^\infty(J_l,L^2)}^{a_2(\eps)} \|\nabla u\|_{L^\infty(J_l,L^2)}^{b_2(\eps)}.
	\end{align*}
	This shows that 
	\[
	\|\scal{\nabla} u\|_{S(L^2,J_l)} \lesssim \|\scal{\nabla} u(t_l)\|_{L^2} + \|\scal{\nabla} u\|_{S(L^2,J_l)} \delta^{\frac{\eps(N+1)}{2(2+\eps)}} C(E,M).
	\]
	Taking $\delta>0$ small enough, we obtain
	\[
	\|\scal{\nabla} u\|_{S(L^2,J_l)} \lesssim \|\scal{\nabla} u(t_l)\|_{L^2} \leq C(E,M).
	\]
	By summing over a finite number intervals $J_l, l=1, \cdots, L$, we prove \eqref{glo-str-bou}. 
	
	We now show the scattering property of global solutions. By the time reversal symmetry, it suffices to consider positive times. By Duhamel formula, we have that
	\[
	e^{-it\Delta} u(t) = u_0 - i \int_0^t e^{-is\Delta} |x|^{-b} |u|^\alpha u ds.
	\]
	Now let $t_2>t_1>0$. By Strichartz estimates,
	\begin{align*}
	\|e^{-it_2\Delta} u(t_2) - e^{-it_1\Delta} &u(t_1)\|_{H^1} \\
	&\lesssim \left\| -i \int_{t_1}^{t_2} e^{-is\Delta} |x|^{-b} |u|^\alpha u ds \right\|_{H^1} \\
	&\lesssim \|\scal{\nabla} (|x|^{-b} |u|^\alpha u)\|_{L^2([t_1,t_2],L^{\frac{2N}{N+2}})} \\
	&\lesssim \sum_{k=0}^1 \||x|^{-b} |\nabla|^k(|u|^\alpha u)\|_{L^2([t_1,t_2],L^{\frac{2N}{N+2}})} + \||x|^{-b-1} |u|^\alpha u\|_{L^2([t_1,t_2],L^{\frac{2N}{N+2}})} \\
	&\lesssim \|\scal{\nabla} u\|_{S(L^2,[t_1,t_2])} \|u\|^{\frac{\eps(N+1)}{2(2+\eps)}}_{L^{N+1}([t_1,t_2], L^{\frac{2N}{N+2}})} \|u\|_{L^\infty([t_1,t_2], L^2)}^{a_1(\eps)} \|\nabla u\|_{L^\infty([t_1,t_2],L^2)}^{b_1(\eps)} \\
	&\mathrel{\phantom{\lesssim}} + \|\scal{\nabla} u\|_{S(L^2,[t_1,t_2])} \|u\|^{\frac{\eps(N+1)}{2(2+\eps)}}_{L^{N+1}([t_1,t_2], L^{\frac{2N}{N+2}})} \|u\|_{L^\infty([t_1,t_2], L^2)}^{a_2(\eps)} \|\nabla u\|_{L^\infty([t_1,t_2],L^2)}^{b_2(\eps)}.
	\end{align*}
	Thanks to \eqref{glo-bou}, \eqref{glo-str-bou} and the conservation of mass and energy, we see that 
	\[
	\|e^{-it_2\Delta} u(t_2)- e^{-it_1\Delta} u(t_1)\|_{H^1} \rightarrow 0 \text{ as } t_1, t_2 \rightarrow +\infty.
	\]
	Hence the limit
	\[
	u_0^+: = \lim_{t\rightarrow +\infty} e^{-it\Delta} u(t) = u_0 - i \int_0^{+\infty} e^{-is\Delta} |x|^{-b} |u|^\alpha u ds
	\]
	exists in $H^1$. Moreover, 
	\[
	u(t) - e^{it\Delta} u_0^+ = i\int_t^{+\infty} e^{i(t-s)\Delta} |x|^{-b} |u|^\alpha u ds.
	\]
	Estimating as above, we show as well that
	\[
	\|u(t) - e^{it\Delta} u_0^+ \|_{H^1} \rightarrow 0 \text{ as } t\rightarrow +\infty.
	\]
	The proof is complete.
	\hfill $\Box$


\begin{thebibliography}{10}
		
		\bibitem{ADM} A. K. Arora, B. Dodson and J. Murphy, {\em Scattering below the ground state for the 2D radial nonlinear Schr\"odinger equation}, preprint \url{arXiv:1906.00515}, 2019.
		
		\bibitem{Berge} L. Berg\'e, Solition stability versus collapse, {\em Phys. Rev. E} 62(3), R3071R3074, 2000.
		
		\bibitem{Campos} L. Campos, Scattering of radial solutions to the inhomogeneous nonlinear Schr\"odinger equation, preprint \url{arXiv:1905.02663}, 2019.
			
		\bibitem{Cazenave} T. Cazenave, {\em Semilinear Schr\"odinger Equations}, Courant Lecture Notes in Mathematics 10, American Mathematical  Society, Courant Institute of Mathematical Sciences, 2003.	
		
		\bibitem{CG} V. Combet and F. Genoud, Classification of minimal mass blow-up solutions for an $L^2$-critical inhomogeneous NLS, {\em J. Evol. Equ.} 16(2):483--500, 2016.
		
		\bibitem{Chen} J. Chen, On a class of nonlinear inhomogeneous Schr\"odinger equation, {\em J. Appl. Math. Comput.} 32:237--253, 2010.
		
		\bibitem{CG} J. Chen and B. Guo, Sharp global existence and blowing up results for inhomogeneous Schr\"odinger equations, {\em Discrete Contin. Dyn. Sys. Ser. B} 8(2):357--367, 2007.
		
		\bibitem{Dinh-weigh} V. D. Dinh, Scattering theory in a weighted $L^2$ space for a class of the defocusing inhomogeneous nonlinear Schr\"odinger equation, preprint \url{arXiv:1710.01392}, 2017.
		
		\bibitem{Dinh-blow}	V. D. Dinh, Blowup of $H^1$ solutions for a class of the focusing inhomogeneous nonlinear Schr\"odinger equation, {\em Nonlinear Anal.} 174:169--188, 2018.
		
		\bibitem{Dinh-scat}	V. D. Dinh, Energy scattering for a class of the defocusing inhomogeneous nonlinear Schr\"odinger equation, {\em J. Evol. Equ.} 19(2):411--434, 2019.
		
		\bibitem{Dinh-repul} V. D. Dinh, On nonlinear Schr\"odinger equations with repulsive inverse-power potentials, preprint \url{arXiv:1812.08405}, 2019.
		
		\bibitem{DM} B. Dodson and J. Murphy, A new proof of scattering below the ground state for the 3D radial focusing cubic NLS, {\em Proc. Amer. Math. Soc.} 145(11):4859--4867, 2017.
		
		\bibitem{Farah} L. G. Farah, Global well-posedness and blow-up on the energy space for the inhomogeneous nonlinear Schr\"odinger equation, {\em J. Evol. Equ.} 16(1):193--208.
		
		\bibitem{FG-3D}	L. G. Farah and C. M. Guzm\'an, Scattering for the radial 3D cubic focusing inhomogeneous nonlinear Schr\"odinger equation, {\em J. Differential Equations} 262(8):4175--4231, 2017.
		
		\bibitem{FG-high} L. G. Farah and C. Guzm\'an, Scattering for the radial focusing INLS equation in higher dimensions, {\em Bull. Braz. Math. Soc.}, 2019 (in press).
		
		\bibitem{FW} G. Fibich and X. P. Wang, Stability of solitary waves for nonlinear Schr\"odinger equations with inhomogeneous nonlinearities, {\em Physica D} 175:96--108, 2003.
		
		\bibitem{GS} F. Genoud and C. A. Stuart, Schr\"odinger equations with a spatially decaying nonlinearity: existence and stability of standing waves, {\em Discrete Contin. Dyn. Syst.} 21(1):137--186, 2008.
		
		\bibitem{Genoud-2d} F. Genoud, A uniqueness result for $\Delta u -\lambda u + V(|x|) u^p =0$ on $\R^2$, {\em Adv. Nonlinear Stud.} 11(3):483--491, 2011.
		
		\bibitem{Genoud} F. Genoud, An inhomogeneous, $L^2$-critical, nonlinear Schr\"odinger equation, {\em Z. Anal. Anwend.} 31(3):283--290, 2012.
				
		\bibitem{Guzman} C. M. Guzm\'an, On well posedness for the inhomogeneous nonlinear Schr\"odinger equation, {\em Nonlinear Anal. Real World Appl.} 37:249--286, 2017.
		
		\bibitem{KM} C. E. Kenig and F. Merle, Global well-posedness, scattering and blow-up for the energy-critical, focusing, non-linear Schr\"odinger equation in the radial case, {\em Invent. Math.} 166(3):645--675, 2006.
		
		\bibitem{LWW} Y. Liu, X. P. Wang and K. Wang, Instability of standing waves of the Schr\"odinger equations with inhomogeneous nonlinearity, {\em Trans. Amer. Math. Soc.} 385(5):2105--2122, 2006.
		
		\bibitem{Merle} F. Merle, Nonexistence of minimal blow-up solutions of equations $iu_t = -\Delta u- k(x) |u|^{\frac{4}{d}} u$ in $\R^N$, {\em Ann. Inst. H. Poincar\'e Phys. Th\'eor.} 64(1):35--85, 1996.
		
		\bibitem{RS} P. Rapha\"el and J. Szeftel, Existence and uniqueness of minimal blow-up solutions to an inhomogeneous mass critical NLS, {\em J. Amer. Math. Soc.} 24(2):471--546, 2011.
		
		\bibitem{Strauss} W. A. Strauss, Existence of solitary waves in higher dimensions, {\em Comm. Math. Phys.} 55(2):149--162, 1977.
		
		\bibitem{Tao} T. Tao, {\em Nonlinear dispersive equations: local and global analysis}, CBMS Regional Conference Series in Mathematics 106, AMS, 2006.
		
		
		\bibitem{TM} I. Towers and B. A. Malomed, Stable (2+1)-dimensional solitions in a layered medium with sign-alternating Kerr nonlinearity, {\em J. Opt. Soc. Amer. B Opt. Phys.} 19(3):537--543, 2002.
		
		
		\bibitem{XZ} {C. Xu and T. Zhao}, A remark on the scattering theory for the 2D radial focusing INLS, preprint \url{arXiv.1908.00743}, 2019.
		
		\bibitem{Yanagida} E. Yanagida, Uniqueness of positive radial solutions of $\Delta u +g(r) u + h(r) u^p =0$ in $\R^N$, {\em Arch. Ration. Mech. Anal.} 115:257--274, 1991.
		
		\bibitem{Zhu} S. Zhu, Blow-up solutions for the inhomogeneous Schr\"odinger equation with $L^2$ supercritical nonlinearity, {\em J. Math. Anal. Appl.} 409:760--776, 2014.
		
	\end{thebibliography}
\end{document}